%% file: main.tex
\documentclass{article}
\usepackage[margin=2.5cm]{geometry}
\usepackage{graphicx} %
\usepackage{amsfonts,amsmath,amssymb,amsthm,mathtools}
\usepackage{tikz}
\usepackage{enumerate}
\usepackage{thm-restate}
\usepackage{hyperref}
\usepackage{authblk}
 \usepackage[normalem]{ulem}

\usepackage{algpseudocode}
\usepackage{algorithm}

\newtheorem{theorem}{Theorem}[section]
\newtheorem*{theorem*}{Theorem}
\newtheorem{claim}[theorem]{Claim}
\newtheorem{proposition}[theorem]{Proposition}
\newtheorem{lemma}[theorem]{Lemma}
\newtheorem*{lemma*}{Lemma}

\theoremstyle{definition}

\newtheorem{definition}[theorem]{Definition}

\foreach \x in {A,...,Z}{%
    \expandafter\xdef\csname m\x\endcsname{\noexpand\mathbf{\x}}
      \expandafter\xdef\csname c\x\endcsname{\noexpand\mathcal{\x}}
}
\newcommand{\0}{\mathbf{0}}

\newcommand{\R}{\mathbb{R}}
\newcommand{\Rp}{\mathbb{R}_{\ge0}}
\newcommand{\Rpp}{\mathbb{R}_{>0}}
\newcommand{\Q}{\mathbb{Q}}

\newcommand{\Dp}{\cD_{>0}^n}
\newcommand{\Pm}{\mathcal{P}}
\newcommand{\Pms}{\mathcal{P}_{*}}
\newcommand{\LCP}{\mathrm{LCP}}
\newcommand{\PT}{\cT}
\newcommand{\hk}{\hat\kappa}
\newcommand{\hks}{\hat\kappa^\star}
\newcommand{\diag}[1]{\operatorname{diag}(#1)}
\newcommand{\length}[1]{\operatorname{length}(#1)}
\newcommand{\supp}{\operatorname{supp}}

\newcommand{\1}{\mathbf{1}}
\newcommand{\pr}[2]{\left\langle #1,#2\right\rangle}

\title{Handicap reduction for linear complementarity problems}
\author[1]{Marianna E.-Nagy\thanks{\tt marianna.eisenberg-nagy@uni-corvinus.hu}}
\author[2]{L{\'{a}}szl{\'{o}} A. V{\'{e}}gh\thanks{{\tt lvegh@uni-bonn.de}. The paper is based on work  during visits  to the Corvinus Institute for Advanced Studies, Corvinus University of Budapest in 2023/24.}}

\affil[1]{Corvinus Centre for Operations Research, Corvinus Institute for Advanced Studies, Corvinus University of Budapest, Hungary}
\affil[2]{Hertz Chair for Algorithms and Optimization, University of Bonn,  Germany}

\date{}

\begin{document}

\maketitle

\begin{abstract}
Linear Complementarity Problems (LCPs) with sufficient matrices form an important subclass of LCPs, and it remains a significant open question whether problems in this class can be solved in polynomial time. Kojima, Megiddo, Noma, and Yoshise gave an Interior Point Algorithm (IPA) in 1991, that can solve LCPs with sufficient matrices in time bounded polynomially in the input size and the so-called handicap number $\hk(\mM)$ of the coefficient matrix $\mM$. However, this value can be exponentially large in the bit encoding length. In fact, no upper bounds were previously known on $\hk(\mM)$. Settling an open question raised in 
de Klerk and E.-Nagy (Math Programming, 2011), we give an exponential upper bound on $\hk(\mM)$ in the bit-complexity of $\mM$.
This is based on a new characterization of sufficient matrices. The new characterization also leads to a simple new proof of V\"aliaho's theorem on the equivalence of sufficient and $\cP^*$-matrices (Linear Algebra and its Applications, 1996).

Noting that one can obtain an equivalent LCP by rescaling the rows and columns by a positive diagonal matrix, we define $\hks(\mM)$ as the best possible handicap number achievable under such rescalings. Our second main result is an algorithm for LCPs with sufficient matrices, where the running time is polynomially bounded in the input size and in the optimized value $\hks(\mM)$. This algorithm is based on the observation that the set of near-optimal row-rescalings forms a convex set. Our algorithm combines the Ellipsoid Method over the set of row rescalings, and an IPA with running time dependent on the handicap number of the matrix. If the IPA fails to solve the LCP in the desired running time, it provides a separation oracle to the Ellipsoid Method to find a better rescaling.
\end{abstract}

\input{intro}

\section{Preliminaries}\label{sec:prelim}

\paragraph{Notation} We let $[n]\coloneqq\{1,\dots,n\}$. 
For $\mM\in\R^{n\times n}$, we let $\mM_i$ denote the $i$-th row of $\mM$ and $\|\mM\|_2$ the spectral norm. For vectors $u,v\in\R^n$, let $u\circ v\in\R^n$ denote the coordinate-wise (Hadamard) product, i.e., $(u\circ v)_i=u_iv_i$. We let $\Dp$ denote the set of $n\times n$ positive diagonal matrices, and $\cD^n$ the set of $n\times n$ nonsingular diagonal matrices.
For $d\in\R^n$, let $\diag{d}\in\R^{n}$ denote the diagonal matrix $\mD\in\R^{n\times n}$ with $\mD_{ii}=d_i$. Let $e^i\in \R^n$ denote the $i$-th unit vector.

\paragraph{Sufficient matrices and $\Pm$-matrices} A matrix $\mM\in \R^{n\times n}$ is a \emph{$\Pm$-matrix} if all of its principal minors are positive. 
A matrix $\mM\in\R^{n\times n}$ is \emph{column} sufficient if $x\circ (\mM x)\le \0_n$  implies $x\circ (\mM x)= \0_n$, and \emph{row sufficient} if $\mM^\top$ is column sufficient. A matrix is \emph{sufficient} if it is both row and column sufficient.
We will use the following lemma, see \cite[Theorem 3.3.4]{cottle2009linear}. 

\begin{lemma}\label{lem:p-reverse}
A matrix $\mM\in\mathbb{R}^{n\times n}$ is a $\Pm$-matrix if and only if $x\circ (\mM x)\le \0_n$ implies $x=\0_n$.
\end{lemma}
Clearly, $\mM$ is a $\Pm$-matrix if and only if $\mM^\top$ is a $\Pm$-matrix. Hence, this lemma implies that all $\Pm$-matrices are sufficient.

For a particular vector $x\in\R^n$, the 
 handicap number of $\mM$ with respect to $x$ is 
 \[
\hk(\mM,x)\coloneqq\begin{cases}
\infty\, &\mbox{if }x^\top \mM x<0 \mbox{ and }\cI_\mM^+(x)=\emptyset \, ,\\
-\frac{\sum_{i\in \cI_\mM^-(x)}x_i (\mM x)_i}{4\sum_{i\in \cI_\mM^+(x)}x_i (\mM x)_i}-\frac{1}{4}\, ,&\mbox{if } x^\top \mM x<0 \mbox{ and } \cI_\mM^+(x)\neq\emptyset\, ,\\
0\, ,&\mbox{if } x^\top \mM x\ge 0\, .
\end{cases}
\]
With this notation, we can equivalently write
\[
\hk(\mM)=\sup_{x\in\R^{n}} \hk(\mM,x) =\sup_{\|x\|\le1} \hk(\mM,x)=\sup_{\|x\|=1} \hk(\mM,x)\,  .
\]

In general, the supremum of $\hat \kappa(\mM,x)$ is not attained at any $x$, but for special matrices an optimal point exists:
\begin{lemma}[{\cite[Theorem 2.3]{Valiaho97}}]\label{lem:P_handicap_x}
If $\mM\in\mathcal{P}\cap\mathbb{R}^{n\times n}$, then there exists a vector $\hat{x}\in\mathbb{R}^n$ such that $\hat\kappa(\mM)=\hat\kappa(\mM,\hat{x})$.
\end{lemma}

We also observe the following simple statement. Note that here we allow $\mD$ to be any nonsingular (not necessarily positive) diagonal matrix.
\begin{lemma}\label{lem:diag-rescale-invariant}
Let $\mM\in\R^{n\times n}$, $\mD\in\cD^n$, and $x\in\R^n$. Then, $\hk(\mM,x)=\hk(\mD \mM \mD,\mD^{-1}x)$. Consequently, $\hk(\mM)=\hk(\mD \mM \mD)$.
\end{lemma}

\paragraph{Encoding length of rational numbers.} For an integer $p\in\mathbb{Z}$, $p\neq 0$, we let $\length{p}=1+\lceil \log_2 |p|\rceil$ denote its binary encoding length. For a rational number represented as $p/q$, we let $\length{p/q}=\length{p}+\length{q}$. For a vector $v$ and matrix $\mM$, we define the encoding length $\length{v}$ and $\length{\mM}$ as the sum of the encoding length of all entries. We use the following simple bounds, see \cite[Section 1.3] {gls}
\begin{lemma}\label{lem:encoding-length}
\begin{enumerate}[(a)]
\item \label{part:length-prod}For $a,b\in\Q$, $\length{a+b}\le 2\length{a}+\length{b}$ and
$\length{ab}\le \length{a}+\length{b}$.
\item \label{part:vec-prod} For vectors $u,v\in\Q^n$, $\length{\pr{u}{v}}\le \length{u}+\length{v}$.
\item  \label{part:det}For a matrix $\mA\in\Q^{n\times n}$, $\length{\det(\mA)}\le 2\length{\mA}$.
\end{enumerate}
\end{lemma}

\paragraph{The principal pivotal transformation}
Let $\mA\in\R^{n\times n}$ and $\cJ\subseteq[n]$ be an index set such that $\mA_{\cJ\cJ}$ is a nonsingular principal submatrix of $\mA$.
Let $\bar\cJ=[n]\setminus\cJ$ denote the complement of $\cJ$. The \emph{principal pivotal transformation} of matrix $\mA$ for $\mA_{\mathcal{J}\mathcal{J}}$ is the matrix
 \[\PT_\mathcal{J}(\mA)\coloneqq
\begin{bmatrix} \mA_{\mathcal{J}\mathcal{J}}^{-1} & -\mA_{\mathcal{J}\mathcal{J}}^{-1}\,
\mA_{\mathcal{J}\bar{\mathcal{J}}} \\
\mA_{\bar{\mathcal{J}}\mathcal{J}}\, \mA_{\mathcal{J}\mathcal{J}}^{-1} &\;\;
\mA_{\bar{\mathcal{J}}\bar{\mathcal{J}}} - \mA_{\bar{\mathcal{J}}\bar{\mathcal{J}}}\,
\mA_{\mathcal{J}\mathcal{J}}^{-1}\, \mA_{\mathcal{J}\bar{\mathcal{J}}}
\end{bmatrix}. 
\]
For a singleton $\cJ=\{j\}$, we use $\PT_{j}$ to denote $\PT_{\{j\}}$.
The main significance of this operation is that if $y=\mA x$ for the vectors $x,y\in\R^n$, then $(x_{\cJ},y_{\bar\cJ})=\PT_{\cJ}(\mA) (y_{\cJ},x_{\bar\cJ})$. That is, it corresponds to exchanging variables $(\mA x)_\mathcal{J}$ and $x_\mathcal{J}$; 
see \cite[Section 2.3]{cottle2009linear}.
It is also easy to see that principal pivoting defines an equivalence relation of $\R^{n\times n}$: if $\mB=\PT_{\cJ}(\mA)$ and $\mC=\PT_{\cK}(\mB)$, then $\mC=\PT_{\cJ\Delta\cK}(\mA)$. 
 Noting that any principal pivot on a matrix $\mM$ leaves the vector $x\circ (\mM x)$ unchanged, we immediately obtain the following.
\begin{lemma}[\cite{guu1995}]\label{lem:principal-sufficient}
The matrix $\mM\in \R^{n\times n}$ is sufficient if and only if every principal pivotal transform of $\mM$ is sufficient. 
\end{lemma}
\paragraph{Sign patterns of sufficient matrices}
We state some well-known properties of sufficient matrices; for completeness, we include their proofs.
\begin{lemma}[\cite{Valiaho97}]\label{lem:sufficient-props}
Let $\mM\in \R^{n\times n}$.
\begin{enumerate}[(i)]
\item If $\mM$ is either row sufficient or column sufficient, then $\mM_{ii}\ge 0$ for all $i\in[n]$. \label{it:diagonal}
\item Assume $\mM$ is column sufficient.
Let $i\in [n]$ such that $\mM_{ii}=0$. Then, $\mM_{ij}\mM_{ji}\le 0$ for all $j\in [n]$, and if $\mM_{ji}\neq0$ then $\mM_{ij}\neq0$.\label{it:opposite-column}
\item Assume $\mM$ is row sufficient.
Let $i\in [n]$ such that $\mM_{ii}=0$. Then, $\mM_{ij}\mM_{ji}\le 0$ for all $j\in [n]$, and if $\mM_{ij}\neq 0$ then $\mM_{ji}\neq 0$. \label{it:opposite-row}
\end{enumerate}
\end{lemma}
\begin{proof}{\bf (a)} If $\mM_{ii}<0$, then $(e^i)\circ (\mM e^i)\lneq \0_n$ and $(e^i)\circ (\mM^\top e^i)\lneq \0_n$, showing that the matrix is neither row nor column sufficient.
{\bf (b)} For a contradiction, assume $\mM_{ii}=0$, and 
$\mM_{ij}\mM_{ji}> 0$. For some $\varepsilon>0$,
consider $x=e^i-\varepsilon\cdot\mathrm{sgn}(\mM_{ij})e^j\in\R^n$. %
Thus, $x_i(\mM x)_i=-|\mM_{ij}|\varepsilon$ and $x_j(\mM x)_j=-|\mM_{ji}|\varepsilon+\mM_{jj}\varepsilon^2$. We have $x_k(\mM x)_k=0$ for $k\notin\{i,j\}$.
For sufficiently small $\varepsilon>0$,  $x\circ(\mM x)\lneq \0_n$ holds. This proves $\mM_{ij}\mM_{ji}\le 0$. If $\mM_{ji}\neq 0$ but $\mM_{ij}=0$, we get a contradiction using the same $x$ as above. Part {\bf (c)} follows by applying part {\bf (b)} for $\mM^\top$.
\end{proof}

\section{Bounding the handicap number}\label{sec:handicap-sec-new}
In this section, we prove Theorem~\ref{thm:kappa-bound}. The key is to 
analyze combinations of two basic solutions to the following polytope. 
For a matrix $\mM\in\R^{n\times n}$, and two disjoint index sets   $\cI^+, \cI^-\subseteq [n]$, $\cI^+\cap \cI^-=\emptyset$, we define the polytope
\[
\cQ(\mM,\cI^+,\cI^-):=\left\{x\in\sigma_n\, |\, \mM_i x\ge 0\;\, \forall i\in \cI^+  ,\;\; \mM_i x\le 0\;\, \forall i\in \cI^- ,\;\; \mM_i x= 0\;\, \forall i\in [n]\setminus(\cI^+\cup \cI^-)\right\}\, ,
\]
where $\sigma_n:=\{x\in\R^n\, |\, x\ge 0\, , \sum_{i=1}^n x_i=1\}$ is the unit simplex.
Note that for any $x\in \cQ(\mM,\cI^+,\cI^-)$ we have  $\cI^+_\mM(x)\subseteq {\cI^+}\subseteq \cI^+_\mM(x)\cup {\cI^0_\mM(x)}^{\phantom{1}}$ and \; $\cI^-_\mM(x)\subseteq \cI^-\subseteq \cI^-_\mM(x)\cup \cI^0_\mM(x)$.
\begin{lemma}\label{lem:bfs-delta}
Let $\mM\in \Q^{n\times n}$ be a rational  matrix of binary encoding length $L$.  
There exists $\delta\ge 2^{-O(n L)}$ and $\Delta \le 2^{O(nL)}$, such that for any $\cI^+, \cI^-\subseteq [n]$, $\cI^+\cap \cI^-=\emptyset$,
and any diagonal matrix $\mD\in\cD^n$ with $\mD_{ii}\in \{\pm 1\}$,
 and every basic feasible solution $v$ to $\cQ(\mD\mM\mD,\cI^+,\cI^-)$, and $i\in[n]$, either $v_i=0$ or $\delta\le |v_i|\le 1$, and either $\mM_i v=0$ or $\delta\le |\mM_i v|\le \Delta$. 
Moreover, for every basic feasible solution $v$, both $v$ and $\mM v$ have encoding length $O(n^2 L)$.
\end{lemma}
\begin{proof}
We assume w.l.o.g.\,$\mD=\mI_n$, noting that the encoding length of $\mD\mM\mD$ is also $L$.
For every basic solution $v$, the vector of nonzero entries $v_B$ arises as the solution to a linear system of the form $\mA v_B=e^k$, where $\mA$ is a nonsingular matrix comprising of a submatrix of $\mM$ and a row of all ones in the $k$-th row.
By Cramer's rule, every entry $v_i>0$ is the ratio of the determinants of two matrices of encoding length $\le L$, thus Lemma~\ref{lem:encoding-length}\eqref{part:det} yields $\length{v_i}=O(L)$, and consequently, $\length{v}=O(nL)$.  Hence, by Lemma~\ref{lem:encoding-length}\eqref{part:vec-prod}, $\length{\mM_i v}=O(nL)$ for all $i\in [n]$, yielding the bounds $\delta$ and $\Delta$.
This also implies $\length{\mM v}=O(n^2 L)$, proving the second part.
\end{proof}
We start by proving  Theorem~\ref{thm:kappa-bound} for the special case of $\Pm$-matrices, where the proof is considerably simpler.
\begin{proposition}\label{thm:kappa-bound-P}
  For a rational matrix $\mM\in\Pm\cap \Q^{n\times n}$ with $L=\length{\mM}$, $\hk(\mM)\le 2^{O(n L)}$.
\end{proposition}
\begin{proof}
Recall that for a $\mathcal{P}$-matrix the handicap is attained (Lemma \ref{lem:P_handicap_x}), namely there exists a vector $x\in\R^n$ such that $\hk(\mM)=\hk(\mM,x)$. Let us pick such a vector $x$ and
 define $\mD\in\cD^n$ with $\mD_{ii}=-1$ if $x_i<0$ and $\mD_{ii}=1$ otherwise; note that $\mD^{-1}=\mD$. By Lemma~\ref{lem:diag-rescale-invariant}, $\hk(\mD\mM\mD,\mD x)=\hk(\mM,x)$; we also have $\mD x\ge0$. We can thus replace $\mM$ by $\mM'=\mD\mM\mD\in\Pm$. For simplicity of notation, we can henceforth assume that $x\ge0$.
The statement is straightforward for $x=\0_n$. If $x\neq\0_n$, then, noting that $\hk(\mM,x)=\hk(\mM,\alpha x)$ for any real number $\alpha\neq 0$, we can normalize  such that $\|x\|_1=1$, i.e., $x\in\sigma_n$. 

Let us fix $\cI^+\coloneqq\cI^+_\mM(x)$  and  $\cI^-:=\cI^-_\mM(x)$, and define $\cQ\coloneqq\cQ(\mM,\cI^+,\cI^-)$. If $\cI^-=\emptyset$ then $\hk(\mM,x)=0$. For the rest of the proof, we assume $\cI^-\neq \emptyset$.
 By the above normalization, $x\in \cQ$.
By Carath\'eodory's theorem, we can write $x=\sum_{j=1}^h \lambda_j v^{(j)}$, where $h\le n+1$, each $v^{(j)}$ is a basic feasible solution in $\cQ$, and $\lambda\ge 0$, $\sum_{j=1}^h \lambda_j=1$. We order the indices such that $\lambda_1\ge\lambda_2\ge\ldots\ge \lambda_h$. 

By Lemma~\ref{lem:p-reverse}, there exists an index $i\in [n]$ such that $v^{(1)}_i (\mM v^{(1)})_i>0$, namely $i\in\cI^+$. By Lemma~\ref{lem:bfs-delta}, it follows that $v^{(1)}_i\ge \delta$ and  $(\mM v^{(1)})_i\ge\delta$ for $\delta\ge 2^{-O(nL)}$.
Note that $\lambda_1\ge 1/h\ge1/(n+1)$ as it is the largest coefficient. Hence, $x_i\ge \lambda_1 v^{(1)}_i\ge \delta/(n+1)$ and $(\mM x)_i\ge \lambda_1 (\mM v^{(1)})_i\ge \delta/(n+1)$, yielding $x_i (\mM x_i)>\delta^2/(n+1)^2$.
Further, $x^\top \mM x\ge -\|\mM\|_2\ge - O( 2^{L})$ by the normalization $\|x\|_2\le\|x\|_1= 1$. Thus, $\hk(\mM)=\hk(\mM,x)\le 2n^3\|\mM\|_2/\delta^2\le 2^{O(n L)}$.
\end{proof}
The following lemma will be needed in the proofs of Theorem~\ref{thm:kappa-bound} and Theorem~\ref{thm:valiaho}.
\begin{lemma}\label{lem:two-basic} Let $\mM\in \R^{n\times n}$ and $x\ge 0$. Then, there exist basic feasible solutions $u$ and $v$ to
$\cQ(\mM,\cI^+_{\mM}(x),\cI^-_{\mM}(x))$, and $\lambda\in [0,1]$, such that 
$\hk(\mM,u+\lambda v)\ge \hk(\mM,x)/(n+1)^2-1$.
\end{lemma}
\begin{proof}
Let $\cI^-\coloneqq\cI^-_{\mM}(x)$, $\cI^+\coloneqq\cI^+_{\mM}(x)$, and $\cQ\coloneqq\cQ(\mM,\cI^+,\cI^-)$.
If $x=\0_n$ or  $\cI^-(x)=\emptyset$ , then $\hk(\mM,x)=0$ and therefore the statement holds for any $u,v$, and $\lambda$. For the rest, let us assume $x\neq\0_n$ and  $\cI^-\neq\emptyset$.
Noting that $\hk(\mM,x)=\hk(\mM,\alpha x)$ for any $\alpha\neq 0$, we can normalize  such that $\sum_{i=1}^n x_i=1$. 
Hence, $x\in \cQ$.
By Carath\'eodory's theorem, we can write $x=\sum_{j=1}^h \lambda_j v^{(j)}$, where $h\le n+1$, each $v^{(j)}$ is a basic feasible solution in $\cQ$, and $\lambda\ge 0$, $\sum_{j=1}^h \lambda_j=1$. We order the indices such that $\lambda_1\ge\lambda_2\ge\ldots\ge \lambda_h>0$. 
Let us define
\begin{equation}A_{jk}\coloneqq -\sum_{i\in \cI^-} v^{(j)}_i (\mM v^{(k)})_i\quad \mbox{and}\quad  B_{jk}\coloneqq\sum_{i\in \cI^+} v^{(j)}_i (\mM v^{(k)})_i\, .\label{eq:A-B-def-1}\end{equation}
Thus, we can write
\begin{equation}\label{eq:kappa-A-B-1}
1+4\hk(\mM,x)=\frac{\sum_{j,k\in [h]} \lambda_j \lambda_k A_{jk}}{\sum_{j,k\in [h]}\lambda_j \lambda_k B_{jk}}\, .
\end{equation}
Let us pick $j,k\in [h]$, $j\le k$ such that $\lambda_{j} \lambda_{k} A_{j k}$ is maximal ($j=k$ is possible). For $z\coloneqq \lambda_{j} v^{({j})}+\lambda_{k} v^{({k})}$, we see that
\[
\begin{aligned}
1+4\hk(\mM,x)&\le h^2\cdot \frac{\lambda_{{j}}\lambda_{k} A_{{j}{k}}}{\sum_{j',k'\in[h]}\lambda_{j'} \lambda_{k'} B_{j'k'}}\\
&\le
h^2\cdot \frac{-\sum_{i\in\cI^-}\left(\lambda_{j}v_i^{(j)}+\lambda_{k}v_i^{(k)}\right) 
\mM_i \left(\lambda_{j}v^{(j)}+\lambda_{k}v^{(k)}\right) } {\sum_{i\in\cI^+}\left(\lambda_{j}v_i^{(j)}+\lambda_{k}v_i^{(k)}\right) 
\mM_i \left(\lambda_{j}v^{(j)}+\lambda_{k}v^{(k)}\right)}\\
&= h^2(1+4\hk(\mM,z))\, ,
\end{aligned}
\]
and consequently, $\hk(\mM,z)\ge \hk(\mM,x)/(n+1)^2-1$. Let us now set 
\[
u\coloneqq v^{(j)}\, ,\quad v\coloneqq v^{(k)}\, ,\quad\mbox{and}\quad \lambda\coloneqq \frac{\lambda_k}{\lambda_j}\, .
\]
By the choice $j\le k$, we have $\lambda_j\ge\lambda_k$, implying $\lambda\in[0,1]$.
Then, $\hk(\mM,u+\lambda v)=\hk(\mM,z)$, and the statement follows.
\end{proof}

We next derive a bound on $\hk(\mM,z)$ where $z=u+\lambda v$ for basic feasible solutions to a polytope of the form $\cQ(\mM,\cI^+,\cI^-)$.
\begin{lemma}\label{lem:u-v-cases}
Let $\mM\in\R^{n\times n}$. Then, there is a finite value $C(\mM)\in\R$ such that 
for any  $\cI^+, \cI^-\subseteq [n]$ such that $\cI^+\cap \cI^-=\emptyset$, any 
two basic feasible solutions $u$ and $v$ to $\cQ(\mM,\cI^+,\cI^-)$, and $\lambda\in [0,1]$, at least one of the following hold: 
\begin{enumerate}[(a)]
\item\label{it:C-C} $\hk(\mM,u+\lambda v)\le C(\mM)$,
\item\label{it:C-u-v} $u\circ \mM u=\0_n$ and $u\circ (\mM v) + v\circ (\mM u)\lneq \0_n$,
\item \label{it:C-u}$u\circ \mM u\lneq \0_n$,
\item\label{it:C-v} $v\circ \mM v\lneq \0_n$.
\end{enumerate} 
If $\mM \in \Q^{n\times n}$ is a rational  matrix, then this holds for $C(\mM)= 2^{O(n \length{\mM})}$.
\end{lemma}
\begin{proof}
Let $0<\delta\le \Delta<\infty$ be chosen such that for any basic feasible solution $x$ to $\cQ(\mM,\cI^+,\cI^-)$ for any choice of $\cI^+$, $\cI^-$, and $i\in[n]$, either $x_i=0$ or $\delta\le x_i\le \Delta$, and either $(\mM x)_i=0$ or $\delta\le (\mM x)_i\le \Delta$. Since there are finitely many choices of $\cI^+$, $\cI^-$, and finitely many basic feasible solutions in each case, we can choose $\delta$ positive and $\Delta$ finite. Let us define $C(\mM)\coloneqq 4n\Delta^2/\delta^2$; we show that this is suitable choice for \eqref{it:C-C}.

Let $z\coloneqq u+\lambda v$. If $z\circ (\mM z)\ge \0_n$, then we get \eqref{it:C-C} by $0\le C(\mM)$.
For the rest of the proof, we assume that $\cI^-_\mM(z)\neq 0$.
Since $\cI^+_\mM(z)\subseteq \cI^+\subseteq \cI^+_\mM(z)\cup \cI^0_\mM(z)$ and $\cI^-_\mM(z)\subseteq \cI^-\subseteq \cI^-_\mM(z)\cup \cI^0_\mM(z)$, we can write 
 \begin{equation}\label{eq:u-v-lambda-t-1}
1+4\hk(\mM,z)= \frac{-\sum_{i\in \cI^-} u_i (\mM u)_i +\lambda \left( u_i (\mM v)_i+  v_i (\mM u)_i\right) +\lambda^2 v_i (\mM v)_i }
 {\sum_{i\in \cI^+} u_i (\mM u)_i +\lambda \left( u_i (\mM v)_i+  v_i (\mM u)_i\right) +\lambda^2 v_i (\mM v)_i }\, .
 \end{equation}
 By the above choices, the numerator is at most $4n\Delta^2$. If $\cI^+_{\mM}(u)\neq \emptyset$, then the denominator is at least $\delta^2$. Thus, the above expression is at most $4n\Delta^2/\delta^2=C(\mM)$, showing the bound $\hk(\mM,z)\le C(\mM)$ of case \eqref{it:C-C}.

Let us assume henceforth $\cI^+_{\mM}(u)= \emptyset$, i.e., $u\circ \mM u\le \0_n$. If $u\circ \mM u\lneq\0_n$, then we obtain case \eqref{it:C-u}. Thus, let us assume $u\circ \mM u=\0_n$. From \eqref{eq:u-v-lambda-t-1}, we obtain 
 \begin{equation}\label{eq:u-v-lambda-t-2}
1+4\hk(\mM,z)= \frac{-\sum_{i\in \cI^-}  u_i (\mM v)_i+  v_i (\mM u)_i +\lambda v_i (\mM v)_i }
 {\sum_{i\in \cI^+} u_i (\mM v)_i+  v_i (\mM u)_i +\lambda v_i (\mM v)_i }\, .
 \end{equation}
 If $\sum_{i\in \cI^+} u_i (\mM v)_i+  v_i (\mM u)_i>0$, then again the denominator is at least $\delta^2$, while the numerator is at most $3n\Delta^2$, leading to case \eqref{it:C-C}. Assume now $\sum_{i\in \cI^+} u_i (\mM v)_i+  v_i (\mM u)_i= 0$, i.e., $u\circ (\mM v) + v\circ (\mM u)\leq \0_n$. If equality does not hold, then we obtain case \eqref{it:C-u-v}. Otherwise, if $u\circ (\mM v) + v\circ (\mM u)=\0_n$, then \eqref{eq:u-v-lambda-t-2} further becomes
  \begin{equation}\label{eq:u-v-lambda-t-3}
1+4\hk(\mM,z)= \frac{-\sum_{i\in \cI^-} v_i (\mM v)_i }
 {\sum_{i\in \cI^+}  v_i (\mM v)_i }\, .
 \end{equation}
 We assumed $\cI^-_\mM(z)\neq 0$ at the beginning; at this point, this implies $\cI^-_\mM(v)\neq 0$, i.e., the numerator is strictly positive.
Similarly as in the first step of the proof, if $\cI^+_{\mM}(v)\neq \emptyset$ then the right-hand side is at most $n\Delta^2/\delta^2\le C(\mM)$, leading to case~\eqref{it:C-C}. Otherwise, if $\cI^+_{\mM}(v)= \emptyset$, then we get $v\circ\mM v\lneq \0_n$ as in
 case~\eqref{it:C-v}. 

For the last part, let $\mM$ be a rational matrix of binary encoding length $L$. By Lemma~\ref{lem:bfs-delta}, we can choose $\delta\ge 2^{-O(nL)}$ and $\Delta \le 2^{O(n L)}$, and therefore $C(\mM)=4n\Delta^2/\delta^2=2^{O(n L)}$.
\end{proof}

We are ready to prove Theorem~\ref{thm:kappa-bound}.
\kappabound*
\begin{proof}
Let $x\in\R^{n}$ denote a solution such that $\hk(\mM,x)\ge \hk(\mM)/2$. 
As in the proof of Proposition~\ref{thm:kappa-bound-P}, we can assume that $x\ge \0_n$.
Lemma~\ref{lem:two-basic} guarantees the existence of basic feasible solutions $u$ and $v$ to 
$\cQ(\mM,\cI^+_{\mM}(x),\cI^-_{\mM}(x))$, and $\lambda\in [0,1]$, such that 
$\hk(\mM,u+\lambda v)\ge \hk(\mM,x)/(n+1)^2-1\ge \hk(\mM)/(2(n+1)^2)-1$.
We now use Lemma~\ref{lem:u-v-cases} for $\mM$. Case~\eqref{it:C-C} implies $\hk(\mM,u+\lambda v)\le 2^{O(nL)}$, and therefore $\hk(\mM)\le  2^{O(nL)}$. Cases~\eqref{it:C-u} and \eqref{it:C-v} immediately contradict the assumption $\mM\in \Pms$. In the remaining case \eqref{it:C-u-v}, we can consider a sequence $z_t=u+\lambda_t v$ for $\lambda_t\to 0$. Using the expression \eqref{eq:u-v-lambda-t-2}, we can see that $\hk(\mM,z_t)\to \infty$, again a contradiction to $\mM\in \Pms$.
\end{proof}

From the proof of Lemma~\ref{lem:u-v-cases}, we can also derive the following statement that will be useful later on.

\begin{lemma}\label{lem:cross-product}
Let $\mM\in \Pms\cap \R^{n\times n}$, $\cI^-,\cI^+\subseteq[n]$, and $u,v\in\cQ(\mM,\cI^+,\cI^-)$. If $u\circ (\mM u)=\0_n$, then 
\[
  (1+4\hk(\mM)) \sum_{i\in \cI^+} u_i (\mM v)_i+  v_i (\mM u)_i \ge -\sum_{i\in \cI^-} u_i (\mM v)_i+  v_i (\mM u)_i\, . 
\]
\end{lemma}
\begin{proof}
We let $z\coloneqq u+\lambda v$ for $\lambda> 0$. Then, \eqref{eq:u-v-lambda-t-2} holds. The statement follows in the limit $\lambda\to 0$ and since $\hk(\mM)\ge \hk(\mM,z)$.
\end{proof}

\section{A new proof of V\"aliaho's theorem}\label{sec:valiaho}
Our goal in this section is to give a simple new proof of Theorem~\ref{thm:valiaho}, a strengthening of V\"aliaho's theorem on the equivalence of $\Pms$ and sufficient matrices \cite{valiaho1996p}. Our main contribution is formulating condition \eqref{it:u-v}, and using this to derive the equivalence of \eqref{it:PM} and \eqref{it:suff}. We restate the  theorem here for convenience. 
\valiaho*

The first lemma shows \eqref{it:PM} implies \eqref{it:suff}. This statement was first shown by Guu and Cottle \cite{guu1995}; we include a simple, direct proof for completeness. %

\begin{lemma}\label{lem:p-star-suff}
Every $\Pms$-matrix is sufficient.
\end{lemma}
\begin{proof}
It follows by definition that every $\Pms$-matrix is column sufficient. For a contradiction, suppose $\mM\in\Pms$ is not row sufficient. Let us pick such an example with $n$ as small as possible, along with a vector $x\in\R^n$ such that $x\circ (\mM^\top x)\lneq \0_n$. Clearly, $n\ge 2$, as for $n=1$ column and row sufficiency are both equivalent to $\mM_{11}\ge 0$.

If $x_i=0$ for some $i\in [n]$, then we get a smaller counterexample by deleting the $i$-th row and $i$-th column of $\mM$. Hence, we may assume $x_i\neq 0$ for all $i\in [n]$. 

Assume next $(\mM^\top x)_i=0$ for some $i\in [n]$. If $\mM_{ii}\neq 0$, then consider a principal pivotal transformation leading to $\overline{\mM}=\PT_i(\mM)$ and $\bar x\in\R^n$ such that $\bar x\circ (\overline{\mM}^\top \bar x)\lneq \0_n$ and $\bar x_i=0$. Here, $\bar x$ is obtained from $x$ by replacing $x_i$ by $(\mM x)_i$. Again, we get a smaller counterexample after removing the $i$-th row and $i$-th column of $\overline\mM$. If $\mM_{ii}=0$ but $\mM_{ji}\neq 0$ for some $j\neq i$, then column sufficiency and Lemma~\ref{lem:sufficient-props}\eqref{it:opposite-column} imply that $\mM_{ij}\neq 0$, and consequently, $\det(\mM_{\cH\cH})\neq 0$ for $\cH=\{i,j\}$. We get a similar contradiction as above for the principal pivot $\PT_{\cH}(\mM)$.

Consider next the case when  $(\mM^\top x)_i=0$ but $\mM_{ji}=0$ for all $j\in [n]$, that is, if the $i$-th column of $\mM$ is $\0_n$. If the $i$-th row is also $\0_n^\top$, then we can again get a smaller counterexample by removing the $i$-th row and $i$-th column. Otherwise, let us pick an index $j\in [n]$ such that $\mM_{ij}\neq 0$. Consider a vector $y\in\R^{n}$ with $y_k=0$ for $k\notin\{i,j\}$, and $y_j=1$ (but $y_i$ arbitrary). For such a vector, $y\circ (\mM y)$ has two nonzero entries, namely $y_j(\mM y)_j=\mM_{jj}$, and $y_i(\mM y)_i=\mM_{ij} y_i$. Consequently, for a sequence of such vectors $y$ with $\mM_{ij} y_i\to -\infty$, we get $\hk(\mM,y)\to \infty$, implying that $\mM\notin\Pms$.

For the rest of the proof, we can assume that all entries of $x$ and $y\coloneqq\mM^\top x$ are nonzero. Thus, $x_i y_i<0$ for all $i\in [n]$. 
\begin{claim}
$\mM_{ij}= y_j/x_i$ for all $i\neq j$, and $\mM_{ii}=-(n-2) y_i/x_i$ for all $i\in [n]$.
\end{claim}
\begin{proof}
First, let us fix $i,j\in [n]$ with $i\neq j$. Let us assume $x_i>0$; the proof is analogous for $x_i<0$. Let us consider the vector $x'$ such that $x'_k=x_k$ for all $k\neq i$, and $ x'_i\ge 0$ is as small as possible subject to the constraint that $(\mM^\top x')_k\ge 0$ whenever $(\mM^\top x)_k>0$, and  $(\mM^\top x')_k\le 0$ whenever $(\mM^\top x)_k<0$. Clearly, $x'_i<x_i$, $x'\circ (\mM^\top x')\le \0_n$, and either $x'$ or $(\mM^\top x')$ has a zero entry. In case $x'\circ (\mM^\top x')$ has a strictly negative entry, such an $x'$ leads to a smaller counterexample or a contradiction as above. Hence, the only possible case is that $x'\circ (\mM^\top x')=\0_n$. Noting that $\mM_{ii}\ge 0$ (Lemma~\ref{lem:sufficient-props}\eqref{it:diagonal}) and $x'_k=x_k$ for $k\neq i$, $x'_i<x_i$, we see that $(\mM^\top x')_i<0$. Therefore, $x'_i=0$ must hold. Since $x'_j=x_j\neq 0$, it follows that $0=(\mM^\top x')_j=(\mM^\top x)_j-(\mM^\top)_{ji} x_i=y_j-\mM_{ij}x_i$. Thus, $\mM_{ij}=y_j/x_i$ follows.

The formula on the diagonal entries is now immediate. For each $i\in [n]$, $y_i=(\mM^\top x)_i=\mM_{ii} x_i +\sum_{j\neq i} (\mM^\top)_{ij} x_j=\mM_{ii} x_i+(n-1) y_i$, leading to the claimed $\mM_{ii}=-(n-2)y_i/x_i$.
\end{proof}
We now derive a contradiction to the column sufficiency of $\mM$. Let us define the vectors $x'\coloneqq 1/y$ and $y'\coloneqq 1/x$, that is, $x'_i=1/y_i$ and $y'_i=1/x_i$ for $i\in[n]$. By the above claim, we can see that $y'=\mM x'$. Thus, $x'\circ (\mM x')=x'\circ y'\lneq \0_n$, since the entries of this vector are $1/(x_i y_i)$.
\end{proof}

\begin{proof}[Proof of Theorem~\ref{thm:valiaho}]
The implication \eqref{it:PM}$\Rightarrow$\eqref{it:suff} was shown in Lemma~\ref{lem:p-star-suff}. Let us next show \eqref{it:suff}$\Rightarrow$\eqref{it:u-v}. For a contradiction, assume that  $\mM$ is sufficient, and yet \eqref{it:u-v} is violated. Take two violating vectors $u$ and $v$ with $v\circ (\mM v)\le \0_n$, $u\circ (\mM v)+ v\circ(\mM u)\le \0_n$. By column sufficiency, 
 $v\circ (\mM v)\le \0_n$ implies $v\circ (\mM v)=\0_n$. Consequently, $u$ and $v$ satisfy
\begin{equation}\label{eq:u-v-restate}
v\circ (\mM v)= \0_n\,\quad \mbox{and}\quad u\circ (\mM v)+ v\circ(\mM u)\lneq \0_n\, .
\end{equation}
Let us denote the index sets
\[\cI\coloneqq\{i\in[n]\,:\,v_i=0\, , (\mM v)_i\neq 0\}\, , \, \cJ\coloneqq\{i\in[n]\,:\,v_i=0\, , (\mM v)_i= 0\}\, ,\,\cK\coloneqq\{i\in[n]\,:\,v_i\neq 0\, , (\mM v)_i= 0\}\, .
\]
We must have $[n]=\cI\cup \cJ\cup \cK$ because of $v\circ (\mM v)=\0_n$.

We claim that both sufficiency as well as \eqref{eq:u-v-restate} are invariant under principal pivotal transformations. For sufficiency this is stated as Lemma~\ref{lem:principal-sufficient}, and for \eqref{eq:u-v-restate}, note that principal pivot transformations leave both vectors $v\circ (\mM v)$ and $u\circ (\mM v)+v\circ (\mM u)$ invariant. 
A principal pivoting $\PT_{\cJ}(\mM)$ replaces $u$ and $v$ by $((\mM u)_{\cJ},u_{\bar\cJ})$ and $((\mM v)_{\cJ},v_{\bar\cJ})$, respectively.

Among all equivalent matrices and vectors under principal pivoting, let us pick $(\mM,u,v)$ such that $\supp(v)$ is as small as possible.

\begin{claim} \label{claim:M_K,JK}
We have $\mM_{\cK,\cJ\cup \cK}=\0_{(|\cJ|+|\cK|)\times|\cK|}$.
\end{claim}
\begin{proof}
First, let us show that $\mM_{kk}=0$ for all $k\in\cK$. Otherwise, if $\mM_{kk}\neq 0$, then we could replace $\mM$ by $\PT_{k}(\mM)$. This transformation implies replacing the nonzero entry $v_k$ in $v$ by $(\mM v)_k=0$, contradicting the choice of $(\mM,u,v)$ such that $\supp(v)$ is minimal.

Next, let $\mM_{kj}\neq 0$ for $k\in \cK$, $j\in \cJ\cup\cK$.  Since $\mM_{kk}=0$, Lemma~\ref{lem:sufficient-props}\eqref{it:opposite-row} implies $\mM_{jk}\neq 0$. Thus, for $\cH=\{j,k\}$, we have $\det(\mM_{\cH\cH})=-\mM_{jk}\mM_{kj}\neq 0$. Therefore, we can consider the principal pivot transform $\PT_{\cH}(\mM)$. This replaces the nonzero entry $v_k$ by $(\mM v)_k=0$, and $v_j$ by $(\mM v)_j=0$; this again contradicts the choice of $(\mM,u,v)$ with $\supp(v)$ minimal.
\end{proof}
Let us now define $\bar u\coloneqq(u_{\cI\cup \cK}, \0_{\cJ})$, that is, replace all entries $u_i$  by 0 for $i\in \cJ$. 
We claim that
\[
\bar u\circ (\mM v)+v\circ (\mM \bar u)= u\circ (\mM v)+v\circ (\mM  u)\lneq \0_n\, . 
\]
Indeed, $\bar u\circ(\mM v)= u\circ(\mM v)$, since $\bar u_\cI=u_\cI$ and $(\mM v)_{\cJ\cup\cK}=\0_{(|\cJ|+|\cK|)\times(|\cJ|+|\cK|)}$. Furthermore, $v \circ(\mM \bar u)= v\circ(\mM u)$, because $v_{\cI\cup\cJ}=\0_{|\cI|+|\cJ|}$ and $(\mM\bar u)_\cK=(\mM u)_\cK$ by Claim \ref{claim:M_K,JK}.

For $\varepsilon>0$, let us now consider the vector $z(\varepsilon)\coloneqq v+\varepsilon \bar u$. We have
\[
z(\varepsilon)\circ (\mM z(\varepsilon))=v\circ(\mM v)+\varepsilon (\bar u\circ (\mM v)+v\circ (\mM \bar u))+\varepsilon^2 \bar u\circ (\mM\bar u)=\varepsilon (\bar u\circ (\mM v)+v\circ (\mM \bar u))+\varepsilon^2 \bar u\circ (\mM\bar u)\, .
\]
We claim that for sufficiently small $\varepsilon>0$, $z(\varepsilon)\circ (\mM z(\varepsilon))\lneq \0_n$, a contradiction to the (column) sufficiency of $\mM$. As shown above, $\varepsilon (\bar u\circ (\mM v)+v\circ (\mM \bar u))\lneq \0_n$. For every $i\in [n]$ such that $\bar u_i (\mM \bar u)_i\le 0$, we get $z(\varepsilon)_i (\mM z(\varepsilon))_i\le 0$, with strict inequality if $\bar u_i(\mM v)_i+v_i(\mM \bar u)_i<0$. Assume now $\bar u_i (\mM \bar u)_i>0$. By the construction of $\bar u$, we must have $i\notin \cJ$. If $i\in \cI$, then 
 $\bar u_i (\mM v)_i+v_i (\mM \bar u)_i=\bar u_i (\mM v)_i<0$, since both $\bar u_i\neq 0$ and $(\mM v)_i\neq 0$. Similarly, if $i\in \cK$, then $\bar u_i (\mM v)_i+v_i (\mM \bar u)_i=v_i (\mM \bar u)_i<0$. Consequently, for sufficiently small $\varepsilon>0$, we have $z(\varepsilon)_i (\mM z(\varepsilon))_i<0$ for all such indices. This completes the proof of $z(\varepsilon)\circ (\mM z(\varepsilon))\lneq \0_n$, and in turn, the proof of \eqref{it:suff}$\Rightarrow$\eqref{it:u-v}.

\medskip

It is left to show \eqref{it:u-v}$\Rightarrow$\eqref{it:PM}. It suffices to show that if $\mM$ is not in $\Pms$, then $u$ and $v$ as in \eqref{it:u-v} exist. 
$\mM\notin\Pms$ is equivalent to the existence of a sequence of vectors $x^{(t)}\in\R^n$, $t=1,2,\ldots$ such that $\lim_{t\to\infty} \hk(\mM,x^{(t)})=\infty$. If $\hk(\mM,x^{(t)})=\infty$ for some $x^{(t)}$, then $x^{(t)}\circ (\mM x^{(t)})\lneq \0_n$. 
Thus, $v=x^{(t)}$ and $u=\0_n$ contradict \eqref{it:u-v}.

 For the rest, we can assume  $0<\hk(\mM,x^{(t)})<\infty$ for all $t$. In particular, $x^{(t)}\neq \0_n$, and we can assume they are normalized as $\|x^{(t)}\|_1=1$. By selecting a subsequence, we can also assume that for every $i\in [n]$, all $x_i^{(t)}$'s have the same sign, and that the partition $(\cI^+_\mM(x^{(t)}),\cI^0_\mM(x^{(t)}),\cI^-_\mM(x^{(t)}))$ is the same for every $t$; let us denote this common partition as $(\cI^+,\cI^0,\cI^-)$, and let $\cQ\coloneqq\cQ(\mM,\cI^+,\cI^-)$. Similarly as in the proof of Proposition~\ref{thm:kappa-bound-P}, we may assume that $x^{(t)}\ge0$ for all $t$.

 We now evoke Lemma~\ref{lem:two-basic} for each $x^{(t)}$ and $\cQ$. Since the number of basic feasible solutions is finite, we can again reduce to a subsequence where the same basic feasible solutions $u$ and $v$ to $\cQ$ give 
 that $\lim_{t\to\infty}\hk(\mM,u+\lambda_t v)\to \infty$ for coefficients $\lambda_t\in [0,1]$.
Thus, in Lemma~\ref{lem:u-v-cases} for $\mM$, case~\eqref{it:C-C} cannot hold.
Case~\eqref{it:C-u-v} shows that $(u,v)$ is as in \eqref{it:u-v}. Similarly, case~\eqref{it:C-u} and case~\eqref{it:C-v} show that $(u,0)$ and $(v,0)$ are respectively as in \eqref{it:u-v}.
\end{proof}

\section{Near-optimal rescalings of a matrix}\label{sec:near-optimal}
In this section, we investigate the optimized handicap number $\hks(\mM)$ defined in Definition~\ref{def:opt-handicap}.
We start by observing that column scalings cannot change the sets $\cI^+_\mM(x)$, $\cI^0_{\mM}(x)$, and $\cI^-_{\mM}(x)$:
\begin{lemma}\label{lem:rescale-preserve}
For a matrix $\mM\in\R^{n\times n}$ and $x\in\R^n$, and for any positive diagonal $\mD\in\Dp$, we have  $\cI^+_{\mD\mM}(x)=\cI^+_{\mM}(x)$, $\cI^0_{\mD\mM}(x)=\cI^0_\mM(x)$, and $\cI^-_{\mD\mM}(x)=\cI^-_\mM(x)$.
\end{lemma}
\begin{proof}
All three statements follow by noting that the  $(\mM x)_i$ and $(\mD\mM x)_i$ have the same signs for all $i\in [n]$. This is because $(\mM x)_i=\mM_i x$, and $(\mD\mM x)_i=(\mD\mM)_i x=\mD_{ii} \mM_i x$.
\end{proof}
For $\mM\in\R^{n\times n}$  and $\tau\ge0$, let us define the set
\[
\cR(\mM,\tau)\coloneqq\left\{d\in \Rp^n\mid
(1+4\tau)\sum_{i\in \cI^+_\mM(x)} d_i x_i(\mM x)_i +\sum_{i\in \cI^-_\mM(x)} d_i x_i(\mM x)_i\ge 0\quad\forall x\in \R^n, d\ge\1_n\right\}\, .
\]
We impose $d\ge \1_n$ to keep the entries strictly positive. Note that for any $\alpha>0$, $\hk\big(\!\diag{d}\mM\big
)=\hk\big(\!\diag{\alpha d}\mM\big)$. Thus, the set $\cR(\mM,\tau)$ is the intersection of a cone and the shifted nonnegative orthant $\{d\in\R^n: d\ge \1_n\}$.
\begin{lemma}\label{lem:rescale-convex}
    For any $\mM\in\R^{n\times n}$  and $\tau\ge0$, $d\in \cR(\mM,\tau)$ if and only if  $d\ge\1_n$ and $\hk\big(\!\diag{d}\mM\big)\le \tau$. Moreover, the set  $\cR(\mM,\tau)$ is convex.
\end{lemma}
\begin{proof}
Convexity follows since
 every constraint is linear in the variables $d_{i}$, and the intersection of any number of halfspaces is a convex set. The first part of the statement follows by the definition of the handicap number and  Lemma~\ref{lem:rescale-preserve}.
\end{proof}
The following lemma asserts the stability of rescalings. 
\begin{lemma}\label{lem:d-close}
Let $d,d'\in \Rpp^n$ and $0<\delta_1\le\delta_2$ such that $\delta_1 d\le d'\le \delta_2 d$. Then, 
\[
\frac{\delta_1}{\delta_2}\left[1+4\hk\big(\!\diag{d}\mM\big)\right]\le 1+4\hk(\diag{d'}\mM)\le \frac{\delta_2}{\delta_1}\left[1+4\hk\big(\!\diag{d}\mM\big)\right]\, .
\]
Consequently, if $\tau\ge 1$ and $\bar d\in\cR(\mM,\tau)$, 
then for every  $d\in\R^n$ with $\|\bar d-d\|\le 1/4$, $d\in\cR(\mM,2\tau)$.
\end{lemma}
\begin{proof}
Let $\mD=\diag{d}$ and $\mD'=\diag{d'}$.
Let us show the upper bound on $1+4\hk(\mD'\mM)$; the lower bound follows analogously.  For any $x\in\R^n$, $\cI^+_{\mD\mM}(x)=\cI^+_{\mD'\mM}(x)$ and $\cI^-_{\mD\mM}(x)=\cI^-_{\mD'\mM}(x)$  by Lemma~\ref{lem:rescale-preserve}; let us use the shorthands $\cI^+$ and $\cI^-$, respectively.
 Then,
\[
1+4\hk(\mD'\mM,x)=\frac{-\sum_{i\in \cI^-}x_i (\mM x)_i d'_i}{\sum_{i\in \cI^+}x_i (\mM x)_i d'_i}\le\frac{-\sum_{i\in \cI^-}x_i (\mM x)_i \delta_2 d_i}{\sum_{i\in \cI^+}x_i (\mM x)_i \delta_1 d_i}= \frac{\delta_2}{\delta_1} (1+4\hk(\mD\mM,x))\, ,
\]
implying the upper bound. The last part of the statement follows since $\bar d\ge \1$ is required for $\bar d\in\cR(\mM,\tau)$, and therefore $\frac{3}{4}\bar d\le d\le \frac{5}{4}\bar d$. Thus, by the first part, $1+4\hk(\diag{d}\mM)\le \frac{5}{3}(1+4\hk(\diag{\bar d}\mM))\le \frac{5}{3}(1+4\tau)<1+4\cdot 2\tau$ since $\tau\ge 1$.
\end{proof}
For the algorithm in Section~\ref{sec:alg}, we will need the following bit-complexity bounds on optimal rescalings; the technical proof of this is deferred to Section~\ref{sec:d-condense}.
\begin{restatable}{theorem}{rescbound}\label{lem:rescaling-bound}
Let $n\ge 3$, $\mM\in\Q^{n\times n}$ with $L=\length{\mM}$, and $\tau\ge 1$ such that $\cR(\mM,\tau)\neq\emptyset$.  Then, there exists 
 a rescaling $\bar d\in\cR(\mM,8n^2\tau)$ with $\bar d_i\in [1,\tau^{2n^2} 2^{O(n^2L)}]$ for all $i\in [n]$.
 \end{restatable}

\subsection{Example of large optimized handicap number}\label{sec:large-kappa}

\largekappahat*
\begin{proof}
 Since all principal minors of the matrix $\mM$ are positive, $\mM$ is a $\cP$-matrix. Therefore, it is also sufficient.
We now turn to the bound on $\hks(\mM)$.
Consider any rescaling $\mD$ with diagonal entries $(d_1,d_2,d_3)$. Since scaling up all entries does not change $\hk(\mD\mM)$, we may assume $d_1,d_2,d_3\ge 1$. Thus, 
\[
\mD\mM=\begin{pmatrix}
d_1 & \alpha d_1 & -d_1 \\
-d_2 & d_2 & \alpha d_2 \\
\alpha d_3 & -d_3 & d_3
\end{pmatrix}
\]
Assume first that $d_1\ge d_2$, and consider the vector $x=(1,-1,0)$. Then, $x_1(\mM x)_1=d_1(1-\alpha)\le d_2(1-\alpha)<0$, and $x_2(\mM x)_2=2d_2>0$, $x_3(\mM x)_3=0$. This implies  
\[
1+4\hk(\mD\mM,x)\ge \frac{\alpha-1}2\, , 
\]
and consequently $\hk(\mD\mM,x)\ge \frac{\alpha-3}{8}$ as claimed
Similarly, if $d_2\ge d_3$, then we get the same bound for $x=(0,1,-1)$, and if $d_3\ge d_1$, then for $x=(-1,0,1)$. Since at least one of these three cases must hold, we get $\hk(\mD\mM)\ge \frac{\alpha-3}{8}$.
\end{proof}

\subsection{Bounding the rescaling coefficients}\label{sec:d-condense}

In this section, we prove Theorem~\ref{lem:rescaling-bound}.
Let us start by defining an auxiliary directed graph $G(\mM)=(V,E)$ for a matrix $\mM\in\R^{n\times n}$. The node set is $V=[n]$, and for $i\neq j$, $(i,j)\in E$ if and only if $\mM_{ij}\neq 0$. 

\begin{lemma}\label{lem:strongly-connected}
Let $\mM\in\Q^{n\times n}$ with $L=\length{\mM}$ and $\tau\ge 0$. If $d\in \cR(\mM,\tau)$ and $(i,j)\in E$ (i.e., $\mM_{ij}\neq 0$), then 
\[d_i \le \tau 2^{O(L)}d_j\, \]
\end{lemma}
\begin{proof}
Let us define the vector $x\in \R^n$ with
\[
    x_k\coloneqq\begin{cases} 1\, ,&\mbox{if }k=i\, ,\\
    -\frac{\mM_{ii}+1}{\mM_{ij}}\, ,&\mbox{if }k=j\, ,\\
    0\, ,&\mbox{otherwise}\, .
    \end{cases}
\]
Thus, $x_i (\mM x)_{i}=-1$ and $x_k (\mM x)_{k}=0$ for $k\notin\{i,j\}$. The assumption $d\in\cR(\mM,\tau)$ means that we must have $x_j (\mM x)_j>0$ and moreover
\[
    (1+4\tau) d_j x_j (\mM x)_j+d_i x_i (\mM x)_{i}\ge 0\, ,
\]
implying
\[
    {d_i}{}\le (1+4\tau) d_j x_j (\mM x)_j= (1+4\tau) d_j \left(-\frac{\mM_{ii}+1}{\mM_{ij}}\cdot\mM_{ji}+\mM_{jj}\left(\frac{\mM_{ii}+1}{\mM_{ij}}\right)^2  \right)\le 
    \tau 2^{O(L)}d_j\, .
\]
Here, we use that since $\mM_{ij}\neq 0$, we must have ${\mM}_{ij}>2^{-L}$, and $\|\mM\|_\infty\le 2^L$; we also make use of  Lemma~\ref{lem:encoding-length}.
\end{proof}

\rescbound*
\begin{proof}
Let $d\in \cR(\mM,\tau)$, i.e., $\hk(\diag{d}\mM)\le \tau$. Noting that $\alpha d\in\cR(\mM,\tau)$ for any $\alpha>0$, we can assume that the smallest component of $d$ equals 1. Let us rearrange the indices such that $1=d_1\le d_2\le \ldots \le d_n$. 
Let $K_0\coloneqq \tau 2^{O(L)}$ be the bound guaranteed in  Lemma~\ref{lem:strongly-connected}, 
 and let us partition the index set $[n]$ into consecutive blocks $S_1,S_2,\ldots,S_t$ by starting a new block at every index $\ell\in \{2,\ldots,n\}$ for which $d_{\ell}>K_0 d_{\ell-1}$. 
If $t=1$, then  $d_{\ell}\le K_0 d_{\ell-1}$ for every $\ell=2,\ldots,n$, and therefore $\bar d=d$ satisfies the stronger bound $\bar d_i\in [1,\tau^n 2^{O(n L)}]$ for all $i\in [n]$.

For the rest of the proof, assume $t>1$. 
According to Lemma~\ref{lem:strongly-connected}, $\mM_{ij}=0$ whenever  $i\in S_{\ell}$, $j\in S_{\ell'}$ for $\ell'<\ell$. Thus, $\mM$ has blocks corresponding to $S_1,\ldots,S_t$ with all entries 0 below the blocks. Let us also use the notation
$S_{>\ell}\coloneqq\cup_{\ell'>\ell} S_\ell$ and $S_{\ge\ell}\coloneqq\cup_{\ell'\ge\ell} S_\ell$.

Let us pick $2^{-O(nL)}<\delta<\Delta<2^{O(nL)}$ as in Lemma~\ref{lem:bfs-delta}, and define
\[
    K\coloneqq 20n^7\cdot K_0^{2n} \cdot \frac{\Delta^3}{\delta^6}\le \tau^{2n} 2^{O(nL)}\,  .
\]
We define the modified rescaling $\bar d\in\R^n$ by $\bar d_1=1$ and
\[
  \frac{\bar d_{i+1}}{\bar d_i}=\begin{cases}
  \frac{d_{i+1}}{d_i}\, ,&\quad\mbox{if $i+1$ and $i$ are in the same component $S_\ell$,}\\
  K   \, ,&\quad\mbox{if $i+1$ and $i$ are in different components.}
  \end{cases}
\]
Thus, $\bar d_i\in [1,\tau^{2n^2} 2^{O(n^2 L)}]$ for all $i\in [n]$. As it will be used later, we note that
\begin{equation}\label{eq:layer-bound}
\bar d_i\ge \frac{\bar d_j}{K_0^n}\, , \quad \forall \ell\in [t]\, ,\, \forall i\in S_{\ge \ell}\, ,\, \forall j\in S_{\ell}\, ,
\end{equation}
because $\bar d_{j+1}\le K_0\bar d_j$ whenever $j,j+1\in S_{\ell}$ by the definition of the layers, and since $|S_{\ell}|\le n$.

 The rest of the proof shows that for any diagonal matrix $\mH\in\cD^n$ with $\mH_{ii}\in \{\pm 1\}$,  any $\cI^-,\cI^+\subseteq [n]$, and any two basic feasible solutions $u,v$ to $\cQ(\mH\mM\mH,\cI^+,\cI^-)$, and $\lambda\in [0,1]$, we have
\begin{equation}\label{eq:u-v-rescaled}
\hk(\diag{\bar d}\mM, \mH u+\lambda \mH v)\le 2\tau\, .
\end{equation}
This inequality for $\mH=\mI_n$, along with 
Lemma~\ref{lem:two-basic}, $\tau\ge 1$ and $n\ge 3$, guarantees that 
$\hk(\diag{\bar d}\mM,x)\le (n+1)^2(2\tau+1)<8n^2\tau$ for any $x\ge0$. For general $x\in\R^n$, let $\mH_{ii}=1$ if $x_i\ge 0$ and $\mH_{ii}=1$ if $x_i<0$. Then, $\mH x\ge 0$, and thus \eqref{eq:u-v-rescaled} implies $\hk(\diag{\bar d}\mM,\mH x)\le (n+1)^2(2\tau+1)<8n^2\tau$ for any $x\in\R^n$.

To prove \eqref{eq:u-v-rescaled}, we can assume without loss of generality that $\mH=\mI_n$. For each $\ell \in [k]$, let $\cI^-_\ell\coloneqq \cI^-\cap S_\ell$, $\cI^-_{>\ell}\coloneqq \cI^-\cap S_{>\ell}$, $\cI^+_\ell\coloneqq \cI^+\cap S_\ell$, $\cI^+_{>\ell}\coloneqq \cI^+\cap S_{>\ell}$.
Let $x=u+\lambda v\ge 0$. Note that $\cI^+_{\mM}(x)=\cI^+_{\mM}(u)\cup \cI^+_{\mM}(v)\subseteq \cI^+$ and $\cI^-_{\mM}(x)=\cI^-_{\mM}(u)\cup \cI^-_{\mM}(v)\subseteq \cI^-$, and that these inclusions may be strict.

We show that for any $\ell \in [k]$,
\begin{equation}\label{eq:block-statement}
   (2+ 4\tau) \sum_{i\in \cI^+_\ell}\bar d_i x_i (\mM x)_i +\frac{2}{n} \sum_{i\in \cI^+_{>\ell}}\bar d_i x_i (\mM x)_i\ge -\sum_{i\in \cI^-_\ell}\bar d_i x_i(\mM x)_i
\end{equation}
Summing up these inequalities for all $\ell\in [k]$, we get 
\[
   (4+ 4\tau)   \sum_{i\in \cI^+}\bar d_i x_i (\mM x)_i \ge -\sum_{i\in \cI^-}\bar d_i x_i (\mM x)_i\, ,
\]
which implies \eqref{eq:u-v-rescaled} since $4+4\tau<1+8\tau$ by $\tau\ge 1$.

Let us also use the notation $u^{(\ell)}\in \R^n$ for $\ell \in [k]$ defined as 
\[
    u^{(\ell)}_{i}\coloneqq \begin{cases} u_i\, ,& \mbox{if}\ \ i\in S_{\ge \ell}\, ,\, \\
    0\, ,& \mbox{otherwise}\, ,
    \end{cases}
\]
i.e., setting  all entries outside  $S_{\ge \ell}$ to 0. We analogously define  $v^{(\ell)}\in \R^n$ and $x^{(\ell)}\coloneqq u^{(\ell)}+\lambda v^{(\ell)}$.
 Note that the block structure of the matrix implies that
\begin{equation}\label{eq:restrict}
x^{(\ell)}_i=x_i\, ,\quad 
 (\mM x^{(\ell)})_i =  (\mM x)_i \quad \forall i\in S_{\ge\ell}\, ,
\end{equation}
If $x_i (\mM x)_i=0$ for all $i\in \cI^+_{>\ell}$, then $\hk(\diag{d} \mM,x^{(\ell)})\le \hk(\diag{d}\mM)\le\tau$ and \eqref{eq:restrict} imply
\[
     (1+4\tau) \sum_{i\in \cI^+_\ell} d_i x_i (\mM x)_i \ge  -\sum_{i\in \cI^-_\ell} d_i x_i (\mM x)_i\, .
\]
Now,  \eqref{eq:block-statement} follows by noting that $\bar d_i=\alpha_\ell d_i$ for all $i\in S_{\ell}$ for some $\alpha_\ell>0$.
The rest of the proof distinguishes  two cases.

\paragraph{Case I: $\lambda=0$ or $\lambda \in [\delta^2/(3n^2\Delta K_0^n),1]$.} Lemma~\ref{lem:bfs-delta} implies that for each $i\in [n]$, either $x_i=0$ or $x_i\in [\delta^3/(3n^2\Delta K_0^n),2]$, and either $(\mM x)_i=0$ or $|(\mM x)_i|\in [\delta^3/(3n^2\Delta K_0^n),2\Delta]$. 
Note that these bounds are applicable both for $\lambda=0$ (in that case, with stronger guarantees $x_i\in \{0\}\cup [\delta,1]$ and $|(\mM x)_i|\in \{0\}\cup [\delta,\Delta]$) and for the case $\lambda\in [\delta^2/(3n\Delta K_0^n),1]$.

We already covered the case when $x_k(\mM x)_k=0$ for all $k\in \cI^+_{>\ell}$. Hence, assume $x_k (\mM x)_k>0$ for some $k\in \cI^+_{>\ell}$. For any $j\in S_\ell$, we have $\bar d_k\ge K \bar d_j$, and therefore
\begin{equation}\label{eq:from-upper-layer}
   \frac{1}{n} \bar d_k x_k (\mM x)_k \ge K \bar d_j\cdot\frac{\delta^6}{9n^5\Delta^2 K_0^{2n}}  \ge  2n^2 \bar d_j \Delta \ge  \bar d_j n |x_j (\mM x)_j|\, , \quad\forall j\in S_{\ell}
\end{equation}
and  \eqref{eq:block-statement} follows by picking $j\in \cI^-_\ell$ with $\bar d_j |x_j (\mM x)_j|$ maximal, and bounding
\[
-\sum_{i\in \cI^-_\ell}\bar d_i x_i(\mM x)_i\le n  \bar d_j |x_j (\mM x)_j|\le  \frac{1}{n} \bar d_k x_k (\mM x)_k < \frac{2}{n} \sum_{i\in \cI^+_{>\ell}}\bar d_i x_i (\mM x)_i\, .
\]

\paragraph{Case II: $\lambda \in (0,\delta^2/(3n^2\Delta K_0^n))$.}
Note that
\[
    x_i(\mM x)_i=u_i(\mM u)_i+ \lambda v_i (\mM u)_i+ \lambda u_i (\mM v)_i+ \lambda^2 v_i (\mM v)_i\, .
\]
First, assume $u_i (\mM u)_i>0$ for some  $i\in \cI^+_{>\ell}$. Recall that $\cI^+_{\mM}(x)=\cI^+_{\mM}(u)\cup \cI^+_{\mM}(v)\subseteq \cI^+$, and hence $x_i (\mM x)_i\ge u_i (\mM u)_i$. Further, by  Lemma~\ref{lem:bfs-delta}, $u_i (\mM u)_i\ge \delta^2$.
Then, \eqref{eq:block-statement} follows similarly as in \eqref{eq:from-upper-layer}: let us select  $j\in \cI^-_{\ell}$ with $\bar d_j |x_j (\mM x)_j|$ maximal. Using also that $\bar d_k\ge K\bar d_j$,
\[
       \frac{1}{n} \bar d_i x_i (\mM x)_i \ge K \bar d_j\cdot\delta^2  >  4n\bar d_j \Delta \ge n \bar d_j |x_j (\mM x)_j|\ge \sum_{i\in \cI^-} \bar d_i |x_i (\mM x)_i|\, .
\]
For the rest, we distinguish two subcases.
\paragraph{Case II(a):
$u_i (\mM u)_i= 0$ for all $i\in \cI^+_{>\ell}$  but $u_k (\mM u)_k>0$ for some $k\in \cI^+_{\ell}$.} Thus, $\hk(\diag{d} \mM,u^{(\ell)})\le\tau$, \eqref{eq:restrict}, and the fact that $\bar d_j=\alpha_\ell d_j$ for all $j\in S_\ell$ implies
\begin{equation}\label{eq:case-ii-a}
     (1+4\tau) \sum_{i\in \cI^+_\ell} \bar d_i u_i (\mM u)_i \ge  -\sum_{j\in \cI^-_\ell} \bar d_j u_j (\mM u)_j\, .
\end{equation}
Our next goal is to show that 
\begin{equation}\label{eq:u-dominates-lambda}
     \bar d_k u_k (\mM u)_k \ge  -\sum_{j\in \cI^-_\ell} \bar d_j \left(\lambda u_j (\mM v)_j+\lambda v_j (\mM u)_j+ \lambda^2 v_j (\mM v)_j\right)\, .
\end{equation}
To show \eqref{eq:u-dominates-lambda}, recall from \eqref{eq:layer-bound} that for each $i,j\in S_{\ell}$, $\bar d_i\ge \bar d_j/K_0^n$. Using Lemma~\ref{lem:bfs-delta} and the Case II upper bound on $\lambda$,  we get that for any $j\in \cI^-_\ell$,
\[
\frac{1}{n}   \bar d_k u_k (\mM u)_k \ge \frac{1}{n}\bar d_i \delta^2 \ge \bar d_j \frac{\delta^2}{nK_0^n}    >  \bar d_j \cdot 3\Delta  \lambda \ge \bar  d_j \left(\lambda |u_j (\mM v)_j|+\lambda |v_j (\mM u)_j|+ \lambda^2 |v_j (\mM v)_j|\right) \, .
\]
Then, \eqref{eq:u-dominates-lambda} follows by suming up over all $j\in\cI^{-}_\ell$.
We can now derive \eqref{eq:block-statement} using \eqref{eq:case-ii-a} and \eqref{eq:u-dominates-lambda} as
\[
\begin{aligned}
   (2+ 4\tau) \sum_{i\in \cI^+_\ell}\bar d_i x_i (\mM x)_i +\frac{2}{n} \sum_{i\in \cI^+_{>\ell}}
   \bar d_i x_i (\mM x)_i &\ge (2+ 4\tau) \sum_{i\in \cI^+_\ell}\bar d_i u_i (\mM u)_i\\
   &\ge (1+ 4\tau) \sum_{i\in \cI^+_\ell}\bar d_i u_i (\mM u)_i + \bar d_k u_k (\mM u)_k \\
   &\ge  -\sum_{j\in \cI^-_\ell} \bar d_j u_j (\mM u)_j -\sum_{j\in \cI^-_\ell} \bar d_j \left(\lambda u_j (\mM v)_j+\lambda v_j (\mM u)_j+ \lambda^2 v_j (\mM v)_j\right)\\
   &= -\sum_{j\in \cI^-_\ell}\bar d_j x_j(\mM x)_j
\end{aligned}
\]
This completes the analysis of Case II(a).

\paragraph{Case II(b):
$u_i (\mM u)_i= 0$ for all $i\in \cI^+_{\ge \ell}$.} This implies $u^{(\ell)}\cdot (\mM u^{(\ell)})\le \0_n$; by sufficiency of $\mM$, we get that  $u^{(\ell)}\cdot (\mM u^{(\ell)})= \0_n$. In particular,
\[
    x_i(\mM x_i)=\lambda v_i (\mM u)_i+ \lambda u_i (\mM v)_i+ \lambda^2 v_i (\mM v)_i\, , \quad\forall i\in S_{\ge \ell}\, .
\]
If 
$u^{(\ell)}\circ (\mM v^{(\ell)})+  v^{(\ell)}\circ (\mM u^{(\ell)})\le \0_n$, then Lemma~\ref{lem:cross-product} for  $u^{(\ell)}$ and $v^{(\ell)}$ implies that $u^{(\ell)}\circ (\mM v^{(\ell)})+  v^{(\ell)}\circ (\mM u^{(\ell)})= \0_n$. Therefore, $x^{(\ell)}\circ(\mM x^{(\ell)})=\lambda^2 v^{(\ell)}\circ (\mM v^{(\ell)})$. The proof of \eqref{eq:block-statement} in this case follows similarly as Case I for $\lambda=0$, i.e., when  $x^{(\ell)}\circ(\mM x^{(\ell)})=u^{(\ell)}\circ (\mM u^{(\ell)})$. Namely, for each $i\in S_{\ge \ell}$, we get 
\[ |x_i(\mM x)_i|\in \{0\}\cup \lambda^2[\delta,\Delta]\, .\]
 As we have already covered the case when $x_i(\mM x)_i=0$ for all $i\in \cI^+_{>\ell}$, we can assume that there exists a $k\in  \cI^+_{>\ell}$ with  $x_k(\mM x)_k>0$. Using that for any $j\in S_{\ell}$, $\bar d_k\ge K\bar d_j$, the above bounds give
\[
   \frac{1}{n} \bar d_k x_k (\mM x)_k \ge\frac{K}{n}  \bar d_j\cdot \lambda^2 \delta^2 >n \bar d_j\lambda^2 \Delta \ge  \bar d_j n |x_j (\mM x)_j|\, ,\quad\forall j\in S_{\ell}
\]
and \eqref{eq:block-statement} follows the same way as in Case I.

\medskip

For the rest of the proof we can assume that there exists some $h\in \cI^+_{\ge \ell}$ such that $v_h (\mM u)_h+ u_h (\mM v)_h>0$, and hence, at least $\delta^2$.  If possible, then choose $h\in \cI^+_{> \ell}$.
By \eqref{eq:layer-bound}, we get  $\bar d_h \ge \bar d_j /K_0^n$ for any $j\in S_{\ell}$. By also using the Case II upper bound  $\lambda<\delta^2/(3n^2\Delta K_0^n)$, we get
\[
  \frac{1}{n^2} \bar d_h (v_h (\mM u)_h+ u_h (\mM v)_h)>\frac{\bar d_h \delta^2}{n^2}\ge \frac{ \bar d_j\delta^2}{n^2 K_0^n} > \lambda \bar d_j \Delta \, ,\quad\forall j\in S_\ell\, .
\] 
By summing up over $j\in \cI^{-}_{\ell}$ and multiplying by $\lambda$, we get
\[
  \frac{\lambda}{n} \bar d_h (v_h (\mM u)_h+ u_h (\mM v)_h)\ge  \lambda^2 \Delta \sum_{j\in \cI^-_{\ell}} \bar d_j\ge \lambda^2 \sum_{j\in \cI^-_{\ell}}  \bar d_j |v_j (\mM v)_j|\, ,
\] 
Consequently,
\begin{equation}\label{eq:h-bound}
 \frac{1}{n} \bar d_h x_h (\mM x)_h\ge  \lambda^2 \sum_{j\in \cI^-_{\ell}}  \bar d_j |v_j (\mM v)_j|\, ,
\end{equation}
Next, we show
\begin{equation}\label{eq:block-statement-2}
   (1+ 4\tau) \sum_{i\in \cI^+_\ell}\bar d_i \left(u_i (\mM v)_i+ v_i (\mM u)_i\right) +\frac{1}{n} \sum_{i\in \cI^+_{>\ell}}\bar d_i \left(u_i (\mM v)_i+ v_i (\mM u)_i\right)\ge -\sum_{j\in \cI^-_\ell}\bar d_j\left(u_j (\mM v)_j+ v_j (\mM u)_j\right)\, .
\end{equation}
To verify this, note that 
if we were able to choose $h\in \cI^+_{> \ell}$, then for any $j\in \cI^-_\ell$,
\[
     \frac{1}{n^2} \bar d_h (v_h (\mM u)_h+ u_h (\mM v)_h)>\frac{K\bar d_j}{n^2}\delta^2 > 2 \bar d_j \Delta \ge \bar d_j \left(|u_j (\mM v)_j|+|v_j (\mM u)_j|\right)\, ,
\]
Summing up over $j\in \cI^-_\ell$, we get
\[
    \frac{1}{n} \sum_{i\in \cI^+_{>\ell}}\bar d_i \left(u_i (\mM v)_i+ v_i (\mM u)_i\right)>  \frac{1}{n} \bar d_h (v_h (\mM u)_h+ u_h (\mM v)_h)\ge -\sum_{j\in \cI^-_\ell}\bar d_j\left(u_j (\mM v)_j+ v_j (\mM u)_j\right)\, ,
\]
and therefore \eqref{eq:block-statement-2} holds.
We chose  $h\in \cI^+_{\ell}$ only when $v_i (\mM u)_i+ u_i (\mM v)_i=0$ for all $i\in \cI_{>\ell}$.
In this case, Lemma~\ref{lem:cross-product} for $\diag{d}\mM$, $u^{(\ell)}$ and $v^{(\ell)}$, 
and the fact that $\bar d_j=\alpha_\ell d_j$ for all $j\in S_{\ell}$ yields
\[
    (1+ 4\tau) \sum_{i\in \cI^+_\ell}\bar d_i \left(u_i (\mM v)_i+ v_i (\mM u)_i\right)\ge   -\sum_{i\in \cI^-_\ell}\bar d_i\left(u_i (\mM v)_i+ v_i (\mM u)_i\right)\, ,
\]
completing the proof of  \eqref{eq:block-statement-2}. 
This now implies 
\begin{equation}
   (1+ 4\tau) \sum_{i\in \cI^+_\ell}\bar d_i x_i (\mM x)_i+\frac{1}{n} \sum_{i\in \cI^+_{>\ell}}\bar d_i x_i (\mM x)_i\ge -\lambda\sum_{j\in \cI^-_\ell}\bar d_j\left(u_j (\mM v)_j+ v_j (\mM u)_j\right)\, .
\end{equation}
Adding up with \eqref{eq:h-bound}, which contributes an extra term $1/n$ to either the first or to the second sum, we obtain 
\[
    \begin{aligned}
   \left(1+\frac{1}{n}+ 4\tau\right)& \sum_{i\in \cI^+_\ell}\bar d_i x_i (\mM x)_i+\frac{2}{n} \sum_{i\in \cI^+_{>\ell}}\bar d_i x_i (\mM x)_i\\
   &\ge -\lambda\sum_{j\in \cI^-_\ell}\bar d_j\left(u_j (\mM v)_j+ v_j (\mM u)_j\right)-  \lambda^2 \sum_{j\in \cI^-_{\ell}}  \bar d_j v_j (\mM v)_j=  -\sum_{j\in \cI^-_\ell} x_j(\mM x)_j \,  ,
\end{aligned}
\]
where the last equality uses that $u^{(\ell)}\circ (\mM u^{(\ell)})=\0_n$ in the current case.
This completes the proof of \eqref{eq:block-statement} for Case II(b), and thereby completing the proof of the lemma.\end{proof}

\section{An LCP algorithm with optimized handicap number dependence}\label{sec:alg}

\newcommand{\LCPIPA}{\mathsf{LCPIPA}}
\newcommand{\Eupdate}{\mathsf{EllipsoidUpdate}}

Let us restate the linear complementarity problem \ref{eq:LCP}: for $\mM\in\R^{n\times n}$ and $q\in \R^n$, the goal is to find vectors $x,s\in\R^n$ such that
\begin{equation}\label{eq:LCP-2}\tag{LCP$(\mM,q)$}
-\mM x+s=q\, ,\quad x^\top s =0\, ,\quad x,s\ge \0\, .
\end{equation}

In this section, we prove the following.

\lcpmain*

First of all, we would like to check the feasibility of the LCP problem. Fukuda and Terlaky \cite{FT} proved an alternative theorem for the LCP, which they called LCP Duality Theorem. Before recalling it, let us consider the dual LCP problem (DLCP)
\cite{FT}: find vectors ${u}, \, {z} \in \mathbb{R}^n$, that
satisfy the constraints
\begin{equation}\label{DLCP_prob}\tag{DLCP$(\mM,q)$}
{u}+\mM^\top{z}={0},\quad {q}^\top{z}=-1,\quad {u}^\top {z}={0},\quad
{u},\,{z}\ge{0}.
\end{equation}
We consider the linear relaxation \ref{DLCP_prob} as
\[\mathcal{F}_D=\{({u},{z})\ge{0}\;:\;{u}+\mM^\top{z}={0},\;
{q}^\top{z}=-1\}.\]
The above mentioned theorem by Fukuda and Terlaky \cite{FT} asserts duality between these systems:
\begin{theorem} Let a sufficient matrix $\mM
\in \mathbb{R}^{n \times n}$ and a vector ${q} \in \mathbb{R}^{n}$
be given. Then exactly one of \ref{eq:LCP-2} and \ref{DLCP_prob} has a feasible solution.
\end{theorem}

Later, Ill\'es, E.-Nagy, and Terlaky~\cite{DLCP} observed the following; the proof is immediate by the definition of row sufficiency.
\begin{lemma}\label{lem:DLCP}
Let $\mM$ be a row sufficient matrix. Then, $({u},{z})\in\mathcal{F}_D$ if and only if $({u},{z})$ is a solution to \ref{DLCP_prob}.
\end{lemma}

Combining these two statements, it follows that  the solvability of a sufficient LCP reduces to solving a linear feasibility problem, and thus can be checked in polynomial time. Therefore, from now on, we assume that the LCP has a solution. 
Our first main tool is the following result from \cite{illes2010polynomial}. %

\begin{theorem}[{\cite[Theorem 20]{illes2010polynomial}}]\label{thm:LCP-kappa}
Let $\mM\in \Q^{n\times n}$ and $q\in\Q^n$, and $\rho\ge 0$, and assume  \ref{eq:LCP-2} is feasible. %
 There exists an algorithm whose running time can be bounded polynomially in $\length{\mM,q}$ and linearly in $\rho$ and terminates with one of the following outputs:
\begin{enumerate}[(i)]
\item  A feasible solution $(x,s)\in \Rp^{2n}$ to \ref{eq:LCP-2} whose size is polynomially bounded in $\length{\mM,q}$.
\item A vector $x\in \Q^n$ whose size is polynomially bounded in $\length{\mM,q}$ such that
\begin{equation}\label{eq:certificate}
(1+4\rho)\sum_{i\in \cI^+_\mM(x)}  x_i(\mM x)_i +\sum_{i\in \cI^-_\mM(x)}  x_i(\mM x)_i<0\, ,
\end{equation}
certifying that $\hk(\mM)>\rho$.
\end{enumerate}
\end{theorem}
Thus, the algorithm either solves the LCP in the running time bound applicable under the assumption $\hk(\mM)\le \rho$, or finds a certificate of $\hk(\mM)>\rho$. Note that even in case  $\hk(\mM)>\rho$ the algorithm may terminate with a feasible solution as in \emph{(i)}.
 We denote the algorithm asserted in the theorem as $\LCPIPA(\mM,q,\rho)$. 

\paragraph{The Ellipsoid Method}
We now state the result on our main other tool, the Ellipsoid Method, as described in \cite{gls}. For a convex body $K\subseteq \R^n$, a \emph{strong separation oracle for $K$} takes a vector $y\in\R^n$ as input, and either correctly concludes $y\in K$, or finds a vector $c\in\R^n$ such that $\pr{c}{y}<\min\{\pr{c}{x}\mid x\in K\}$. In an \emph{oracle-polynomial algorithm},
the number of oracle calls and the number of arithmetic operations is polynomial in the input encoding length.

\begin{theorem}\label{thm:ellipsoid}
Let $0<r<R$, and let $K\subseteq\R^n$ be a convex body such that $\|x\|_2\le R$ for all $x\in K$, and assume a strong separation oracle is available for $K$. There exists an 
oracle-polynomial time algorithm that in $O(n^2 \log(R/r))$ oracle calls and additional polynomial time, either finds a feasible solution $x\in K$, or concludes that $K$ does not contain a ball of radius $r$.
\end{theorem}
In more detail, the ellipsoid method maintains in each iteration an ellipsoid of the form
\[
    E(\mB,d)=\{x\in\R^n\,:\, (x-d)^\top \mB (x-d)\le 1\}\, ,
\]
where $\mB\in\R^{n\times n}$ is a positive definite matrix, and $d$ the centre of the ellipsoid. 

The key subroutine $\Eupdate(\mB,d,c)$ takes the current ellipsoid and a direction vector $c\in\R^n$, and finds a new ellipsoid $E(\mB',d')$ that contains the intersection of $E(\mB,d)$ with the halfspace defined by  $\pr{c}{x}\ge \pr{c}{d}$, i.e.,
\begin{equation}\label{eq:ellipsoid-contain}
    E(\mB',d') \supseteq E(\mB,d)\cap \{x\in\R^n\, :\, \pr{c}{x}\ge \pr{c}{d}\}\, .
\end{equation}
The key volumetric progress is guaranteed by
\begin{equation}\label{eq:ellipsoid-progress}
    \mathrm{vol}(E(\mB',d'))\le \mathrm{e}^{-1/(2n)}\mathrm{vol}(E(\mB,d))\, ,
\end{equation}
see Chapter 3, in particular Lemma 3.1.34 in \cite{gls}. Recall that the volume of $E(\mB,d)$ is  $\mathrm{vol}(E(\mB,d))=V_n/\sqrt{\det(\mB)}$, where $V_n$ is the volume of the $n$ dimensional unit ball.

\paragraph{The algorithm}
The proof of Theorem~\ref{thm:LCP-main} relies on
Algorithm~\ref{alg:lcp}, 
 a conjuction of the ellipsoid method and the interior point subroutine $\LCPIPA{}$ as 
  in Theorem~\ref{thm:LCP-kappa}. A description of the algorithm is given in the proof below.

\begin{algorithm}
    \caption{LCP-solve}\label{alg:lcp}
    \raggedright

    \begin{algorithmic}[1]
    \Require{Matrix $\mM\in\Q^{n\times n}$ and $q\in\R^n$ with total encoding length $L$}
    \Ensure{A feasible solution $(x,s)\in\R^{2n}$ to \ref{eq:LCP-2},  or a feasible solution $(u,z)\in\R^{2n}$ to the dual \ref{DLCP_prob}, or the conclusion  $\mM\notin\Pms$} 
    \State Solve the LP feasibility program $\cF_D$. 
    \If{a feasible solution $(u,z)\in\cF_D$ is found}
    \If{$u^\top z=0$} 
    \State \Return $(u,z)$ and terminate.
    \Else 
    \State \Return $\mM\notin\Pms$ and terminate.
    \EndIf\EndIf
     \State $\tau\gets 1$ ; $T\gets 2^{O(nL)}$ as in Theorem~\ref{thm:kappa-bound} ;
     \State $K\gets \tau^{2n^2}2^{O(n^{2} L)}$ as in Theorem~\ref{lem:rescaling-bound} ; $R\gets  \sqrt{n}K$; 
    \While{$\tau<T$}\label{l:main-while}
    \State $\mB\gets \mI_n/R$ ; $d\gets \1_n$ ;
    \While{$\det(\mB)<1/4^{2n}$}\label{l:volume-while}
  \State Call  $\LCPIPA(\diag{d}\mM,\diag{d} q,8n^2\tau)$.
  \If{feasible $(x,s)\in\R^{2n}$ is found}
  \State{ \Return $(x,\diag{1/d}s)$ and terminate.}
  \ElsIf{$x\in\R^n$ is found with $\hk(\diag{d}\mM,x)>5n^2\tau$} 
    \For{$i=1,2,\ldots,n$} \If{$i\in \cI^+(x)$} $c_i\gets (1+32n^2\tau)x_i(\mM x)_i$ \Else~ 
    $c_i\gets x_i(\mM x)_i$ \EndIf\EndFor
    \State $(\mM,d)\gets \Eupdate(\mB,d,c)$ ; 
    \While{$d_i<1$ or $d_i>K$ for some $i\in [n]$}
    \If{$d_i<1$ for some $i\in [n]$} 
     \State $(\mM,d)\gets \Eupdate(\mB,d,e^i)$ ;
     \EndIf
    \If{$d_i>K$ for some $i\in [n]$} 
     \State $(\mM,d)\gets \Eupdate(\mB,d,-e^i)$ ;
     \EndIf
\EndWhile
    \EndIf
\EndWhile
\State $\tau\gets 2\tau$ ;\label{l:tau-double}
\EndWhile
\State\Return{$\mM\notin\Pms$}
    \end{algorithmic}
\end{algorithm}

\paragraph{Proof of Theorem~\ref{thm:LCP-main}}
The proof follows by analyzing Algorithm~\ref{alg:lcp}. We first check if the linear system $\cF_D$ is feasible. If a solution $(u,z)\in\cF_D$ is found and $u^\top z=0$, then we can return this
as a feasible solution to \ref{DLCP_prob}. If $u^\top z\neq 0$, then we obtain a certificate of $\mM\notin\Pms$ in accordance with Lemma~\ref{lem:DLCP}. In both cases, the algorithm terminates.

For the rest, we assume $\cF_D=\emptyset$, and consequently, \ref{eq:LCP} is feasible.
We define the parameters  $T=2^{O(nL)}$ as the upper bound on $\hk(\mM)$ for a sufficient matrix $\mM$ as in Theorem~\ref{thm:kappa-bound}, and $K=\tau^{2n^2}2^{O(n^{2} L)}$ as the upper bound on the rescaling coefficients as in Theorem~\ref{lem:rescaling-bound}.

Throughout, we maintain a guess $\tau$ on the value of $\hks(\mM)$. This is initialized as $1$. The outermost while cycle of the algorithm, starting in line~\ref{l:main-while}, uses the current estimate of $\tau$. At the end of the cycle, we conclude that $\hks(\mM)<\tau$, and double the value of $\tau$. This cycle terminates once $\tau>T$. At this point, noting that $\hks(\mM)\le \hk(\mM)$, we can conclude in accordance with Theorem~\ref{thm:kappa-bound} that $\mM$ is not a sufficient matrix. This outermost while cycle is repeated at most $\log\hks(\mM)$ times, or if the matrix is not sufficient, at most $O(nL)$ times. For each guess $\tau$, the running time is polynomial in $n$ and $L$, and linear in $\tau$. Since we use the $\tau$ values $1,2,2^2,\ldots,2^r$ for some $r\in \mathbb{N}$, the running time bound for the final cycle will dominate the sum of the running time bounds of all cycles (and hence, we do not incur an additional $\log\hks(\mM)$ factor).

We now discuss the execution of the outermost while cycle. For $R=\sqrt{n}K$, we initialize $\mB=\mI_n/R$ and $d=\1_n$, i.e., the initial ellipsoid will be the ball centred around $d$ with radius $R$. Throughout, we use the centre of the ellipsoid $d$ as our rescaling vector, and
maintain the invariant that 
\begin{equation}\label{eq:alg-invariant}
 \text{either }   \cR(\mM,8n^2 \tau)\cap \mE(\mB,d)\text{ contains a ball of radius }r=1/4\quad\text{or}\quad     \cR(\mM, \tau)=\emptyset\, .
\end{equation}
By the choice of $R$, Theorem~\ref{lem:rescaling-bound} guarantees that whenever 
$\cR(\mM, \tau)\neq\emptyset$, then there exists a $\bar d\in \cR(\mM,4n^2 \tau)\cap [1,K]^n$. Then, by Lemma~\ref{lem:d-close}, the ball of radius $1/4$ around $\bar d$ is contained in $\cR(\mM,8n^2 \tau)\cap [1,K+1]^n$. Since $[1,K+1]^n\subseteq \mE(\mB,d)$ for the initial choice $\mB=\mI_n/R$ and $d=\1_n$, this invariant holds at the beginning of the outermost while cycle.

The inner while cycle of the algorithm, starting in line~\ref{l:volume-while}, terminates once $\det(\mB)>1/4^{2n}$, or equivalently, when 
the volume of the current ellipsoid $E(\mB,d)$ drops below the volume of a  ball of radius $r=1/4$. As the invariant \eqref{eq:alg-invariant} is maintained throughout, at the end of this cycle we may conclude $\cR(\mM, \tau)=\emptyset$, or equivalently, $\hks(\mM)>\tau$. Hence, we double the value of $\tau$ at line~\ref{l:tau-double} and restart the while cycle at line~\ref{l:main-while}.

Let us now describe the inner while cycle starting in in line~\ref{l:volume-while}. We first call the interior point algorithm from Theorem~\ref{thm:LCP-kappa} for the rescaled instance $(\diag{d}\mM,\diag{d}q)$ and parameter $\rho=8n^2\tau$. According to Theorem~\ref{thm:LCP-kappa}, the running time is polynomial in the instance encoding length and linear in $\rho$.

Whenever the IPA outputs a feasible solution $(x,s)$, then $(x,\diag{1/d}s)$ is feasible to the original input $(\mM,q)$, and the algorithm terminates. Otherwise, we obtain a vector $x\in\R^n$ as in \eqref{eq:certificate}, i.e., $\hk(\mM,x)>8n^2\tau$. We now update the ellipsoid using the subroutine $\Eupdate{}$, with the cost function $c_i=(1+20n^2\tau)x_i(\mM x)_i$ for $i\in\cI^+(x)$ and $c_i=x_i(\mM x)_i$ otherwise; this represents the separating hyperplane \eqref{eq:certificate}. Since $\pr{c}{d'}>\pr{c}{d}$ must hold for any rescaling $d'\in \cR(\mM,8n^2\tau)$, the update property \eqref{eq:ellipsoid-contain} shows that invariant \eqref{eq:alg-invariant} is maintained after this rescaling.

However, the update may find a new centre
 $d\notin [1,K]^n$, possibly even with negative coordinates. As we  want to maintain our rescaling coordinates in $[1,K]$, we use the separating directions $e^i$ or $-e^i$ for some $i\in [n]$, and keep calling $\Eupdate{}$ until we reach $d\in [1,K]^n$. Clearly, \eqref{eq:alg-invariant} will still hold after such updates.

According to \eqref{eq:ellipsoid-progress}, the total number of calls to the inner while cycle is bounded as  $O(n^2 \log(R/r))=O\left(n^4(L+\log\tau)\right)$, giving a polynomial overall running time.

Furthermore, we need to guarantee that the encoding length of the vector $d$ remains polynomially bounded in $L$ throughout the algorithm. This is also required to guarantee that $(\diag{d}\mM,\diag{d}q)=\mathrm{poly}(n)L$. Thus, 
$\LCPIPA{}$ terminates in polynomial time for the input  $(\diag{d}\mM,\diag{d}q)$. This can be guaranteed by using the polynomial implementation of the Central Cut Ellipsoid Method in \cite[Section 3.3]{gls}. Namely, one can round $\mB$ and $d$ in each iteration to $O(n^2\log(R/r))=O\left(n^4(L+\log\tau)\right)$ digits in the binary expansion. One can obtain the same asymptotic guarantees, with the coefficient $\mathrm{e}^{-1/(5n)}$ in \eqref{eq:ellipsoid-progress}.

\section{Concluding remarks}\label{sec:conclusion}
In this paper, we analyzed the handicap number $\hk{(\mM)}$ of sufficient matrices, and showed that it can be bounded polynomially in the input encoding length. Further, we introduced the notion of the optimized handicap number $\hks{(\mM)}$ that can be achieved by diagonal rescalings, and showed that the set of near-optimal rescalings form a convex set. This enables us to improve the running time dependence of LCP algorithms from $\hk{(\mM)}$ to $\hks{(\mM)}$; the running time improvement may be of exponential magnitude for certain matrices.
However, as  shown in Section~\ref{sec:large-kappa}, $\hks{(\mM)}$ can still be arbitrarily large. Thus, the question of solving LCPs with sufficient matrices in polynomial time remains open in general.

One can also ask for the complexity status of computing or approximating $\hk(\mM)$ and $\hks(\mM)$. While even approximating these quantities seems very difficult, we are not aware of any hardness proofs. Tseng~\cite{tseng2000co} showed that deciding column (or row) sufficiency is co-NP complete. However, these do not immediately imply co-NP completeness of deciding sufficiency. Note that deciding sufficiency is equivalent to deciding whether a finite $\hk(\mM)$  value exists.

The algorithm in Section~\ref{sec:alg} is a black-box combination of the ellipsoid and interior point methods and as such is highly non-practical. It would be desirable to develop more direct and efficient IPAs for LCPs where the running time depends on $\hks(\mM)$ instead of $\hk(\mM)$.

\paragraph{Acknowledgments} The authors would like to thank Tibor Ill\'es for inspiring discussions on the topic, and Haoyuan Ma for his careful reading and suggestions on the manuscript.
L\'aszl\'o V\'egh acknowledges the support of the Corvinus Institute for Advanced Studies, Corvinus University that hosted him as a nonresident senior research fellow in 2023/24.
The research of Marianna E.-Nagy was supported by the J\'anos Bolyai Research Scholarship of the Hungarian Academy of Sciences (2024-2027).

\bibliographystyle{abbrv}
\bibliography{references}
\end{document}

%% file: intro.tex
\section{Introduction}
Consider a Linear Complementarity Problem (LCP) for a matrix $\mM\in\R^{n\times n}$ and $q\in \R^n$ in the form
\begin{equation}\label{eq:LCP}\tag{LCP$(\mM,q)$}
-\mM x+ s=q\, ,\quad {x}^\top{s}=0\, ,\quad x,s\ge \0\, .
\end{equation}

The general problem encompasses several well-known special cases. We can formulate Linear Programming (LP) in a primal-dual embedding with a skew-symmetric matrix $\mM$, and convex quadratic programs with a positive semidefinite matrix. LCPs also capture problems such as Nash equilibria in two-player bimatrix games, stopping problems, market equilibrium problems, and mean payoff games. We refer the reader to the classical monograph by Cottle, Pang, and Stone \cite{cottle2009linear} on the theory and applications of LCP.

The general problem, however, is strongly NP-complete \cite{chung1989}. The exact boundary between polynomially solvable and hard cases of LCP remains a fundamental question. 
Kojima,  Megiddo, Noma,  and Yoshise \cite{kojima1991unified} gave an important general result in this context, by showing that LCPs where $\mM$ has a bounded \emph{handicap number} are solvable in polynomial time.

\paragraph{The handicap number and $\Pms$-matrices}
The handicap number $\hk(\mM)$ can be seen as a relaxation of positive semidefiniteness. Recall that $\mM\in\R^{n\times n}$ is positive semidefinite (PSD) if and only if $x^\top \mM x=\sum_{i=1}^n x_i (\mM x)_i\ge 0$ for all $x\in\R^n$. Note that we do not assume the symmetry of the matrix $\mM$.
For  $x\in\R^n$, let us consider the following partition of the index set according to the sign of $x_i (\mM x)_i=x_i\mM_i x$, where $\mM_i$ denotes the $i$-th row of $\mM$.
\[
\begin{aligned}
\cI_\mM^+(x)\coloneqq\{i\, |\, x_i (\mM x)_i>0\}\, , \quad  \cI_\mM^0(x)\coloneqq\{i\, |\, x_i (\mM x)_i=0\}\, , \quad  \cI_\mM^-(x)\coloneqq\{i\, |\, x_i (\mM x)_i<0\}\, .
\end{aligned}
\]
The \emph{handicap number} of the matrix $\mM\in\R^{n\times n}$ is defined as
\[
\hk(\mM)\coloneqq\min\left\{\kappa\ge 0\; |\; x^\top \mM x + 4\kappa\! \!\sum_{i\in \cI_\mM^+(x)} \! x_i (\mM x)_i \ge 0\quad \forall x\in\R^n\right\}\, .
\]
Thus, $\hk(\mM)=0$ if and only if $\mM$ is PSD.
We let $\Pms(\kappa)$ denote the set of $n\times n$ matrices with $\hk(\mM)\le \kappa$, and $\Pms\coloneqq\cup_{\kappa\ge 0}\Pms(\kappa)$. A matrix $\mM\in\Pms$ is called a \emph{$\Pms$-matrix}. Kojima et al.~\cite{kojima1991unified} gave a \emph{unified interior point method} that can solve $\LCP(\mM,q)$ in $O((1+\hk(\mM))\sqrt{n} L)$ iterations, where $L$ is the total encoding length of the rational input $(\mM,q)$.

Nevertheless, $\hk(\mM)$ may not be polynomially bounded in $L$, and therefore this does not yield a polynomial-time algorithm for $\Pms$-matrices. The first such example was given by de Klerk and E.-Nagy \cite{deklerk2011complexity}; they showed that the following  simple  matrix with $0,\pm1$ entries, called the \emph{Csizmadia matrix}, has exponential handicap number:
\begin{equation}\label{eq:matrix-csizmadia}
\mC_n \coloneqq
\begin{pmatrix}
 1 &  0 & 0 & \cdots & 0\\
-1 &  1 & 0 & \cdots & 0\\
-1 & -1 & 1 &\cdots &0\\
\vdots   & \vdots   & \vdots  & \ddots & \vdots\\
-1&-1&-1&\cdots&1
\end{pmatrix}\, .\end{equation}
This matrix has $\hk(\mC_n)=2^{2n}-\tfrac14$ (see \cite{deklerk2011complexity,e2023sufficient}). The paper \cite{deklerk2011complexity} raised as an open question to bound the handicap number of a rational matrix with encoding length $L$ as $\hk(\mM)\le 2^{\mathrm{poly}(L)}$. The first contribution of this paper is resolving this in the affirmative, see Theorem~\ref{thm:kappa-bound} below.

\paragraph{Optimizing the handicap number}
Even if $\hk(\mM)$ is large, a suitable preconditioning could turn the LCP into a tractable one.
Let $\Dp\subseteq \R^{n\times n}$ denote the set of $n\times n$ positive diagonal matrices.
 Row and column rescalings of $\mM$ turn the LCP to an equivalent problem: for positive diagonal matrices $\mD',\mD\in\Dp$, $(x,s)$ is a solution to \ref{eq:LCP} if and only if $(\mD^{-1} x,\mD' s)$ is a solution to
 $\LCP(\mD' \mM\mD,\mD'q)$. This motivates the following definition:
 \begin{definition}\label{def:opt-handicap}
For a matrix $\mM\in\R^{n\times n}$, the \emph{optimized handicap number} $\hks(\mM)$ is defined as
\[
\hks(\mM)\coloneqq\inf\{\hk(\mD\mM)\mid \mD\in\Dp\}\, .
\]
\end{definition}
The definition only uses row rescalings, but row scalings can achieve the same handicap values as allowing both row- and column scalings, by noting that $\hk(\mD' \mM \mD')=\hk(\mM)$ for $\mD'\in\Dp$ (see Lemma~\ref{lem:diag-rescale-invariant} below); this is immediate from the definition of $\hk$.

For the matrix $\mC_n$ defined in \eqref{eq:matrix-csizmadia}, it turns out that $\hks(C_n)=0$.   Let
\[
\mD' \coloneqq
\begin{pmatrix}
 1 &  0 & 0 & \cdots & 0\\
0 &  1/2 & 0 & \cdots & 0\\
0 & 0 & 1/2^2 &\cdots &0\\
\vdots   & \vdots   & \vdots  & \ddots & \vdots\\
0&0&0&\cdots&1/2^{n-1}
\end{pmatrix}\, , \quad 
\mD \coloneqq
\begin{pmatrix}
 1 &  0 & 0 & \cdots & 0\\
0 & 2 & 0 & \cdots & 0\\
0 & 0 & 2^2 &\cdots &0\\
\vdots   & \vdots   & \vdots  & \ddots & \vdots\\
0&0&0&\cdots&2^{n-1}
\end{pmatrix}
.\]
These scalings yield
\[
\mD'\mC_n\mD =
\begin{pmatrix}
 1 &  0 & 0 & \cdots & 0\\
-1/2 &  1 & 0 & \cdots & 0\\
-1/2^2 & -1/2 & 1 &\cdots &0\\
\vdots   & \vdots   & \vdots  & \ddots & \vdots\\
-1/2^{n-1}&-1/2^{n-2}&-1/2^{n-3}&\cdots&1
\end{pmatrix}\, .\]
The rescaled matrix is diagonally dominant and therefore positive semidefinite. Consequently, $\hk(\mD'\mC_n\mD)=0$. 
One can similarly show that $\hks(\mM)=0$ for every triangular matrix $\mM$ with positive diagonal entries.

The question of whether $\hks(\mM)=0$ is equivalent to deciding whether a matrix can be turned PSD by a suitable rescaling. This latter class of matrices is already appeared in \cite{deklerk2011complexity} and was denoted by $\mathcal{H}$. De Klerk and E.-Nagy proved that the membership problem for set $\mathcal{H}$ can be formulated as a semidefinite feasibility problem (see Lemma 5 in \cite{deklerk2011complexity}), therefore it can be solved by an interior point algorithm (IPA) in polynomial time with $\varepsilon$-precision.
Let us state here a simplified SDP feasibility problem (a linear matrix inequality) for the sake of completeness:
\[ \mM\in\mathcal{H}\quad\Longleftrightarrow\quad \exists \; y\in \mathbb{R}^{n\times n}_{\ge 0}:\quad \mT_0+\sum_{i=1}^n y_i\mT_i\succeq 0\, , \]
where $\mT_0=\mM+\mM^\top$ and $\mT_i=\diag{e^i}\mM+\mM^\top\diag{e^i}, i=1\dots,n$.\footnote{$\diag{d}\in\Dp$ denotes the diagonal matrix with entries $d$.} For any feasible $y$, the diagonal matrix $\mD$ with the diagonal elements $\mD_{ii}=y_i+1$, $i=1,\dots,n$ defines a suitable rescaling for matrix $\mM$, that is, $DM$ will be a PSD matrix.

However, the complexity of computing $\hks(\mM)$ is not clear; even the complexity of computing of $\hk(\mM)$,  or even deciding finiteness of $\hk(\mM)$ is open, although these problems are likely hard. Despite this, our second contribution, Theorem~\ref{thm:LCP-main} below, asserts that \ref{eq:LCP} can be solved in time polynomial in  $\hks(\mM)$ and the encoding length of $(\mM,q)$, improving the dependence on $\hk(\mM)$ in \cite{kojima1991unified} to $\hks(\mM)$.

\subsection{Our contributions}

 Our first result resolves an open question in \cite{deklerk2011complexity}. The binary encoding length is formally defined in Section~\ref{sec:prelim}.
\begin{restatable}{theorem}{kappabound}\label{thm:kappa-bound}
  For a rational matrix $\mM\in\Pm_*\cap \Q^{n\times n}$ with $L=\length{\mM}$, $\hk(\mM)\le 2^{O(n L)}$.  
\end{restatable}
Showing a $2^{\mathrm{poly}(n) L}$ bound is easy for many matrix condition numbers. Such bounds are known for condition numbers such as
the maximum ratio between singular values, maximum subdeterminants, the Dikin--Stewart--Todd measure $\bar\chi(\mM)$ \cite{vavasis1996primal} or the circuit imbalance measure  \cite{dadush2023scaling}. These arguments ultimately reduce to relating the condition numbers to the entries of the inverse of nonsingular square submatrices of $\mM$.
Analogous arguments suffice to show Theorem~\ref{thm:kappa-bound} for the special case of $\cP$-matrices, i.e., matrices where all principal minors are positive. This argument is given as Proposition~\ref{thm:kappa-bound-P} and we now sketch the proof. We consider $\hat x\in\R^n$ that is  optimal to the program defining $\hk(\mM)$; we may normalize it as $\|\hat x\|_1=1$. Let us define a polytope $\cQ$ containing all points $y$ where $y_i$ and $(\mM y)_i$ have the same signs as $\hat x_i$ and $(\mM \hat x)_i$ respectively, for each index $i$, and $\|y\|_1=1$. We can write $\hat x$ as a convex combination of at most $n+1$ extreme points of $\cQ$; let $v^{(1)}$ be the point with the largest coefficient. The $\cP$-matrix property guarantees that $v^{(1)}_i (\mM v^{(1)})_i>0$ for at least one index (see Lemma~\ref{lem:p-reverse}). This positive entry will be $2^{-O(n L)}$ due to standard bit-complexity arguments on basic solutions. This already suffices to derive the bound on
$\hk(\mM)$.

The proof becomes significantly more involved for matrices $\mM\in\Pms\setminus\cP$. For $u,v\in\R^n$, let $u\circ v\in\R^n$ denote the Hadamard product. We start the proof as above, but it is possible that
$v^{(1)}\circ (\mM v^{(1)})=\0_n$, i.e.,
$v^{(1)}_i (\mM v^{(1)})_i=0$ for all $i$. In fact, $v^{(k)}\circ (\mM v^{(k)})=\0_n$ may happen for all extreme points $v^{(k)}$ with large coefficients in the combination.
Theorem~\ref{thm:kappa-bound} follows by showing that there are two basic solutions $u=v^{(k)}$ and $v=v^{(\ell)}$, and $\lambda>0$, such that the handicap value corresponding to $u+\lambda v$ is within a factor $n^2$ from the optimum value in the problem defining $\hk(\mM)$ (see Lemma~\ref{lem:two-basic}). However, these critical basic solutions $u=v^{(k)}$ and $v=v^{(\ell)}$ may have arbitrarily low coefficients in the combination giving $\hat x$, necessitating a more careful argument.
 
 The same Lemma~\ref{lem:two-basic} also leads to a new equivalent characterization of $\Pms$-matrices, and in turn to a simple new proof of  V\"aliaho's classical result asserting that the class $\Pms$ is the same as the class of sufficient matrices. Let us now define this matrix class. 
\begin{definition}\label{def:sufficient}
A matrix $\mM\in\R^{n\times n}$ is \emph{column} sufficient if $x\circ (\mM x)\le \0_n$  implies $x\circ (\mM x)= \0_n$, and \emph{row sufficient} if $\mM^\top$ is column sufficient. A matrix is \emph{sufficient} if it is both row and column sufficient.
\end{definition}
In the equivalence of sufficient and $\Pms$-matrices, the only trivial direction is that every $\Pms$-matrix is column sufficient. The row sufficiency of $\Pms$-matrices was shown by Guu and Cottle in \cite{guu1995}, and the converse  direction that all sufficient matrices are $\Pms$ was shown by 
V\"aliaho in \cite{valiaho1996p}. We give a simple direct proof of the first direction (Lemma~\ref{lem:p-star-suff}). The new proof of the converse direction is via a third, equivalent characterization as stated next.
\begin{restatable}{theorem}{valiaho}\label{thm:valiaho}
Let $\mM\in\R^{n\times n}$. Then, the following are equivalent:
\begin{enumerate}[(A)]
\item\label{it:PM} $\mM$ is $\Pms$.
\item\label{it:suff} $\mM$ is sufficient.
\item\label{it:u-v} For any vectors $u,v\in \R^n$ such that $v\circ (\mM v)\le \0_n$, $u\circ (\mM v)+ v\circ(\mM u)\le \0_n$, we must have  $v\circ (\mM v)+u\circ (\mM v)+ v\circ(\mM u)= \0_n$.
\end{enumerate}
\end{restatable}
To get an intuition of part~\eqref{it:u-v}, take vectors $u,v\in \R^n$ such that $v\circ (\mM v)\le \0_n$, $u\circ (\mM v)+ v\circ(\mM u)\le \0_n$.
 If $v\circ (\mM v)\lneq\0_n$, then $v$ is a certificate that $\mM$ is not column sufficient. Assume $v\circ (\mM v)=\0_n$ and $u\circ (\mM v)+ v\circ(\mM u)\lneq \0_n$; let $\gamma\coloneqq-(u^\top \mM v+ v^\top \mM u)>0$. Then, for $z\coloneqq v+\varepsilon u$, we have $\sum_{i\in \cI_\mM^+(z)} z_i (\mM z)_i\le \varepsilon^2 \sum_{i\in \cI_\mM^+(u)}u_i (\mM u)_i$, while $\sum_{i\in \cI_\mM^-(z)}z_i (\mM z)_i\le \varepsilon (u^\top \mM v+ v^\top \mM u)+\varepsilon^2 u\circ \mM u=\varepsilon(-\gamma+\varepsilon u\circ \mM u)$. Thus, $|\sum_{i\in \cI_\mM^-(z)}z_i (\mM z)_i|/\sum_{i\in \cI_\mM^+(z)}z_i (\mM z)_i\ge (\gamma-\varepsilon |u\circ \mM u|)/(\varepsilon \sum_{i\in \cI_\mM^+(u)}u_i (\mM u)_i)\to \infty$ as $\varepsilon\to 0$, showing that $\hk(\mM)=\infty$, that is, $\mM$ is not in $\Pms$. Our proof reveals that whenever a matrix is not in $\Pms$, then one can identify either a certificate that it is not column sufficient of the form $v\circ (\mM v)\lneq \0_n$, or a direction witnessing the unboundedness of $\hk(\mM)$. 

\medskip

We now turn to the main algorithmic contribution of the paper.
\begin{restatable}{theorem}{lcpmain}\label{thm:LCP-main}
Let $\mM\in \Q^{n\times n}$ and $q\in\Q^n$. Then, there exists an
algorithm whose running time can be bounded polynomially in $\length{\mM,q}$ and linearly in $\hks(\mM)$, and  returns one of the following outputs:
\begin{enumerate}[(i)]
\item A feasible solution to \ref{eq:LCP}. 
\item A dual certificate of infeasibility.
\item The conclusion that $\mM\notin\Pms$.
\end{enumerate}
\end{restatable}

As already noted, we do not have a preconditioning algorithm to find $\mD\in \Dp$ such that $\hk(\mD\mM)$ is approximately the optimal value $\hks(\mM)$. Instead, we use a combination of the ellipsoid method and an IPA subroutine to gradually find a suitable column scaling. We note that the algorithm does not need to know the value of $\hk(\mD\mM)$.

The crucial observation is that for any $\tau\ge0$, the set of positive diagonal  matrices $\mD\in\Dp$ such that $\hk(\mD\mM)\le \tau$ is convex (Lemma~\ref{lem:rescale-convex}). This follows by observing that for any $\mD\in\Dp$, 
$\cI_{\mD\mM}^-(x)=\cI_\mM^-(x)$ and $\cI_{\mD\mM}^+(x)=\cI_\mM^+(x)$ hold. Moreover, we show that---assuming there exists a rescaling with $\hk(\mD\mM)\le \tau$---there exists a ball $\cB^*\subseteq \Rpp^n$  of radius $r=1/4$, contained in a ball of radius $R=\tau^{2n^2}2^{O(n^2 L)}$, such that for any $d\in \cB^*$, $\hk(\diag{d}\mM)=\rho$ for $\rho=O(n^2\hks(\mM))$ (Theorem~\ref{lem:rescaling-bound} and Lemma~\ref{lem:d-close}).  Thus, the ellipsoid method (see \cite{gls}) would be a natural approach to find a rescaling in $\cB^*$, assuming we can find a separation oracle.

Algorithm~\ref{alg:lcp} in Section~\ref{sec:alg}
maintains a guess $\tau$ on $\hks(\mM)$, initialized as $\tau=1$. We maintain an ellipsoid centered at a point $d\in\Rp^n$,   initialized as $d=\1_n$, and we consider the rescaled LCP instance $(\diag{d}\mM,\diag{d}q)$. We set the parameter $\rho=O(n^2\tau)$, and 
use an IPA for \ref{eq:LCP} for the rescaled system. As noted by Ill\'es, E.-Nagy, and Terlaky \cite{illes2010polynomial}, the unified IPA from \cite{kojima1991unified} terminates with a solution to the LCP in $O((1+\rho)\sqrt{n} L)$ steps as long as all step directions used in the algorithm are consistent with the hypothesis $\mM\in \Pms(\rho)$. That is, either the IPA makes the desired progress in each step, or a vector $z\in\R^n$ is found that testifies $\hk(\mD\mM)>\rho$.

In the latter case, we terminate the IPA, and change our ellipsoid. We use the vector $z$ to define a separating hyperplane between the current centre $d$ and the set of rescalings $d'$ for which $\hk(\diag{d'}\mM)\le \rho$. If the IPA does not succeed, then in a polynomial number of times the volume of the ellipsoid drops below the volume of a ball of radius $1/4$, at which point we conclude that there is no rescaling such that  $\hk(\diag{d'}\mM)\le \tau$. If this happens, we double our guess $\tau$ and restart.

The algorithm described here is not practical as it in effect multiplies the time bounds of the ellipsoid and IPAs. Still, it conceptually improves the running time dependence from $\hk(\mM)$ to $\hks(\mM)$. 
Our approach is loosely inspired by the result of Dadush, Huiberts, Natura, and V\'egh \cite{dadush2023scaling} that improves the running time of IPAs for exact LP from the logarithm of the circuit imbalance measure of the constraint matrix to the optimized version of this measure. However, in contrast to our approach, \cite{dadush2023scaling} is able to efficiently find a suitable rescaling as a preconditioning, as well as develop a scaling invariant IPA. In contrast, we resort to the less efficient approach via the ellipsoid method.

\medskip

As already noted, the optimized handicap $\hks(\mM)$ can be arbitrarily better than $\hk(\mM)$. Our algorithm is polynomial whenever $\hks(\mM)$ is polynomially bounded. Our next example shows that this is not always the case. 
\begin{restatable}{proposition}{largekappahat}\label{lem:large-kappa-hat}
For any $\alpha\ge 3$, the matrix
\begin{equation}\label{eq:matrix-large-handicap}
\mM\coloneqq\begin{pmatrix}
1 & \alpha & -1 \\
-1 & 1 & \alpha \\
\alpha & -1 & 1
\end{pmatrix}
\end{equation}
is sufficient, with $\hks(\mM)\ge \frac{\alpha-3}{8}$.
\end{restatable}

\paragraph{Organization} The rest of the paper is structured as follows. Section~\ref{sec:prelim} introduces some basic concepts and tools on encoding lengths and on sufficient matrices. Section~\ref{sec:handicap-sec-new} shows Theorem~\ref{thm:kappa-bound}, the polynomial bound on $\hk(\mM)$. The characterization developed in this section is then used in Section~\ref{sec:valiaho} to derive a new proof of V\"aliaho's theorem. Section~\ref{sec:near-optimal} introduces the optimized handicap number and studies its properties, giving in particular a polynomial bound on near-optimal rescaling coefficients. Section~\ref{sec:alg} shows Theorem~\ref{thm:LCP-main}, by presenting an LCP algorithm whose running time is polynomially bounded in the input size and the optimized handicap number. Section~\ref{sec:conclusion} concludes with some open questions.